\documentclass[11pt]{article}
\usepackage{amsmath,amssymb,amsthm,amscd}
\usepackage{xcolor}
\usepackage[all]{xy}
\newtheorem{theorem}{Theorem}[section]
\newtheorem{lemma}[theorem]{Lemma}
\newtheorem{proposition}[theorem]{Proposition}
\newtheorem{corollary}[theorem]{Corollary}

\theoremstyle{plain}

\theoremstyle{definition}
\newtheorem{definition}[theorem]{Definition}
\newtheorem{remark}[theorem]{Remark}

\numberwithin{equation}{section}

\renewcommand{\theenumi}{(\roman{enumi})}
\renewcommand{\labelenumi}{\textup{(\theenumi)}}

\title{Reciprocal Cuntz--Krieger algebras} 

\author{Kengo Matsumoto \\
Department of Mathematics \\
Joetsu University of Education \\
Joetsu, Niigata 943-8512, Japan
\and
Taro Sogabe \\
Department of Mathematics \\
Kyoto University \\
Kyoto,  606-8502, Japan}

\begin{document}


\maketitle

\date{}

\def\det{{{\operatorname{det}}}}

\begin{abstract}
Reciprocality in Kirchberg algebras 
is a duality between  strong extension groups and $\operatorname{K}$-theory groups.
We describe a construction of the reciprocal dual algebra $\widehat{\mathcal{A}}$
for a Kirchberg algebra $\mathcal{A}$ 
with finitely generated $\operatorname{K}$-groups via K-theoretic duality for extensions.
In particular, we may concretely construct and realize the reciprocal algebra 
$\widehat{\mathcal{O}}_A$ for simple Cuntz--Krieger algebras $\mathcal{O}_A$
in several different ways.
As a result, 
the algebra $\widehat{\mathcal{O}}_A$ 
is realized as a unital simple purely infinite universal $C^*$-algebra 
generated by a family of partial isometries subject to certain operator relations.
We will finally study gauge actions on the reciprocal algebra 
$\widehat{\mathcal{O}}_A$ and prove that 
there exists an isomorphism between the fundamental groups
$\pi_1({\operatorname{Aut}}({\mathcal{O}}_A))$ and
$\pi_1({\operatorname{Aut}}(\widehat{\mathcal{O}}_A))$
preserving their gauge actions.
\end{abstract}



\tableofcontents



\newcommand{\Ker}{\operatorname{Ker}}
\newcommand{\sgn}{\operatorname{sgn}}
\newcommand{\Ad}{\operatorname{Ad}}
\newcommand{\ad}{\operatorname{ad}}
\newcommand{\orb}{\operatorname{orb}}
\newcommand{\rank}{\operatorname{rank}}

\def\Re{{\operatorname{Re}}}
\def\det{{{\operatorname{det}}}}
\newcommand{\K}{\operatorname{K}}

\newcommand{\sqK}{\operatorname{K}\!\operatorname{K}}

\newcommand{\bbK}{\mathbb{K}}
\newcommand{\N}{\mathbb{N}}
\newcommand{\bbC}{\mathbb{C}}
\newcommand{\R}{\mathbb{R}}
\newcommand{\Rp}{{\mathbb{R}}^*_+}
\newcommand{\T}{\mathcal{T}}
\newcommand{\bbT}{\mathbb{T}}

\newcommand{\Z}{\mathbb{Z}}
\newcommand{\Zp}{{\mathbb{Z}}_+}
\def\AF{{{\operatorname{AF}}}}

\def\TorZ{{{\operatorname{Tor}}^\Z_1}}
\def\Ext{{{\operatorname{Ext}}}}
\def\Exts{\operatorname{Ext}_{\operatorname{s}}}
\def\Extw{\operatorname{Ext}_{\operatorname{w}}}
\def\Ext{{{\operatorname{Ext}}}}
\def\Free{{{\operatorname{Free}}}}

\def\Ks{\operatorname{K}^{\operatorname{s}}}
\def\Kw{\operatorname{K}^{\operatorname{w}}}

\def\OA{{{\mathcal{O}}_A}}
\def\ON{{{\mathcal{O}}_N}}
\def\OAT{{{\mathcal{O}}_{A^t}}}
\def\OAI{{{\mathcal{O}}_{A^\infty}}}
\def\OAIn{{{\mathcal{O}}_{A^\infty_n}}}
\def\OAInone{{{\mathcal{O}}_{A^\infty_{n+1}}}}

\def\OalgAI{{{\mathcal{O}}^{\operatorname{alg}}_{A^\infty}}}

\def\OAIN{{{\mathcal{O}}_{A^\infty_N}}}
\def\OATI{{{\mathcal{O}}_{A^{t \infty}}}}

\def\PAI{{{\mathcal{P}}_{A^\infty}}}
\def\FAI{{{\mathcal{F}}_{A^\infty}}}
\def\IAI{{{\mathcal{I}}_{A^\infty}}}
\def\IAIn{{{\mathcal{I}}_{A^\infty}^n}}

\def\OSA{{{\mathcal{O}}_{S_A}}}

\def\TA{{{\mathcal{T}}_A}}
\def\TAn{{{\mathcal{T}}_{A_n}}}
\def\TAnone{{{\mathcal{T}}_{A_{n+1}}}}
\def\TAN{{{\mathcal{T}}_{A_N}}}

\def\whatA{{\widehat{\A}}}
\def\whatB{{\widehat{\B}}}

\def\TN{{{\mathcal{T}}_N}}

\def\TAT{{{\mathcal{T}}_{A^t}}}

\def\TB{{{\mathcal{T}}_B}}
\def\TBT{{{\mathcal{T}}_{B^t}}}

\def\A{{\mathcal{A}}}
\def\B{{\mathcal{B}}}
\def\C{{\mathcal{C}}}
\def\D{{\mathcal{D}}}
\def\O{{\mathcal{O}}}
\def\OaA{{{\mathcal{O}}^a_A}}
\def\OB{{{\mathcal{O}}_B}}
\def\OTA{{{\mathcal{O}}_{\tilde{A}}}}
\def\F{{\mathcal{F}}}
\def\G{{\mathcal{G}}}
\def\FA{{{\mathcal{F}}_A}}
\def\PA{{{\mathcal{P}}_A}}
\def\PI{{{\mathcal{P}}_\infty}}
\def\OI{{{\mathcal{O}}_\infty}}

\def\OalgI{{{\mathcal{O}}^{\operatorname{alg}}_\infty}}

\def\calI{\mathcal{I}}

\def\calP{\mathcal{P}}
\def\calQ{\mathcal{Q}}
\def\calR{\mathcal{R}}
\def\calC{\mathcal{C}}
\def\calD{\mathcal{D}}
\def\calM{\mathcal{M}}

\def\bbC{{\mathbb{C}}}

\def\U{{\mathcal{U}}}
\def\OF{{{\mathcal{O}}_F}}
\def\DF{{{\mathcal{D}}_F}}
\def\FB{{{\mathcal{F}}_B}}
\def\DA{{{\mathcal{D}}_A}}
\def\DB{{{\mathcal{D}}_B}}
\def\DZ{{{\mathcal{D}}_Z}}

\def\End{{{\operatorname{End}}}}

\def\Ext{{{\operatorname{Ext}}}}
\def\Hom{{{\operatorname{Hom}}}}

\def\Tor{{{\operatorname{Tor}}}}

\def\Max{{{\operatorname{Max}}}}
\def\Max{{{\operatorname{Max}}}}
\def\max{{{\operatorname{max}}}}
\def\KMS{{{\operatorname{KMS}}}}
\def\Per{{{\operatorname{Per}}}}
\def\Out{{{\operatorname{Out}}}}
\def\Aut{{{\operatorname{Aut}}}}
\def\Ad{{{\operatorname{Ad}}}}
\def\Inn{{{\operatorname{Inn}}}}
\def\Int{{{\operatorname{Int}}}}
\def\det{{{\operatorname{det}}}}
\def\exp{{{\operatorname{exp}}}}
\def\nep{{{\operatorname{nep}}}}
\def\sgn{{{\operatorname{sign}}}}
\def\cobdy{{{\operatorname{cobdy}}}}
\def\Ker{{{\operatorname{Ker}}}}
\def\Coker{{{\operatorname{Coker}}}}
\def\Im{{\operatorname{Im}}}

\def\ind{{{\operatorname{ind}}}}
\def\Ind{{{\operatorname{Ind}}}}
\def\id{{{\operatorname{id}}}}
\def\supp{{{\operatorname{supp}}}}
\def\co{{{\operatorname{co}}}}
\def\scoe{{{\operatorname{scoe}}}}
\def\coe{{{\operatorname{coe}}}}
\def\I{{\mathcal{I}}}
\def\Span{{{\operatorname{Span}}}}
\def\event{{{\operatorname{event}}}}
\def\Proj{{{\operatorname{Proj}}}}
\def\S{\mathcal{S}}

\def\calK{\mathcal{K}}

\def\PI{\mathcal{P}_{\infty}}
\def\PAI{{{\mathcal{P}}_{A^\infty}}}

\def\whatOA{\widehat{\mathcal{O}}_A}
\def\whatOAT{\widehat{\mathcal{O}}_{A^t}}

\def\sAI{\sigma_{A^\infty}}
\def\swhatA{\sigma_{\widehat{A}}}
\def\OalgAI{{{\mathcal{O}}^{\operatorname{alg}}_{A^\infty}}}

\def\wtO{\widetilde{O}}
\def\wtA{\widetilde{A}}
\def\wtAI{\widetilde{A}_{\infty}}
\def\wttAI{\widetilde{A}^t_{\infty}}
\def\E{\mathcal{E}}
\def\opE{\operatorname{E}}

\def\gATI{{\gamma_{A^{t \infty}}}}


\section{Introduction and Preliminaries}
\subsection{Introduction}

Non-commutative analogue of the classical Spanier--Whitehead duality 
in topology was introduced and studied by Kaminker--Schochet \cite{KS}.
It was called the Spanier--Whitehead $\K$-duality and 
inspired by $\K$-theoretic duality of Cuntz--Krieger algebras 
studied by Kaminker--Putnam \cite{KP}.
Kaminker and Putnam investigated the duality between 
the $\K_0$-group 
$\K_0(\OA) = \Z^N/(I-A^t)\Z^N$ and 
weak extension group $\Extw(\OA) = \Z^N/(I-A)\Z^N$
of the Cuntz--Krieger algebra $\OA$ from $\sqK$-theoretic view point.
The Spanier--Whitehead $\K$-dual of a separable $C^*$-algebra $\A$
 is written as $D(\A)$. 
The $C^*$-algebra $D(\A)$ for $\A$,
which is unique up to $\sqK$-equivalence, 
but not unique up to isomorphism,
  exists  
if $\K_*(\A)$ are finitely generated (\cite{KS}).

In \cite{Sogabe2022}, the second named author 
 proved that
for a Kirchberg algebra $\A$ with finitely generated $\K$-groups, 
there exists a unique Kirchberg 
algebra $\B$ with finitely generated $\K$-groups
such that 
$\A \sim_{\sqK} D(C_\B)$
and
$\B \sim_{\sqK} D(C_\A),$
where 
$C_\A$ and $C_\B$ are the mapping cone algebras for the unital 
embeddings
$u_\A: \mathbb{C}\rightarrow \A$ and $u_\B: \mathbb{C}\rightarrow \B$
respectively,
and $\sim_{\sqK}$ means $\sqK$-equivalence.
The $C^*$-algebra $\B$ is said to be reciprocal to $\A$,
and written as $\whatA$.
The operation $\A \rightarrow \whatA$ is regarded as a duality which satisfies 
$\widehat{\widehat{\A}} \cong \A.$
 The duality may be considered as a duality between
the strong extension groups  $\Exts$ and the $\K$-groups
such that $\Exts^i(\whatA) \cong \K_{i+1}(\A)$,
whereas the Spanier--Whitehead $\K$-duality  
is a duality between
the weak extension groups  $\Extw$ and the $\K$-groups
such that $\Extw^i(D(\A)) \cong \K_{i+1}(\A)$.
We have to emphasize that the algebra 
$\whatA$ is uniquely determined  by $\A$ up to isomorphism,
whereas 
the Spanier--Whitehead $\K$-dual
$D(\A)$ is uniquely determined  by $\A$ up to $\sqK$-equivalence.
The reciprocality was introduced in studying the homotopy groups 
$\pi_i(\Aut(\A))$
of the automorphism groups 
$\Aut(\A)$
of Kirchberg algebras $\A$ and bundles of $C^*$-algebras.
In \cite{Sogabe2022}, it was proved that 
$\pi_i(\Aut(\A)) \cong \pi_i(\Aut(\whatA))$
and also 
 two Kirchberg algebras $\A, \B$ satisfy 
$\pi_i(\Aut(\A)) \cong \pi_i(\Aut(\B))$
if and only if $\A$ is isomorphic to $\B$, or
$\A$ is isomorphic to $\widehat{B}$.
By using the fact that the conditions
$\A$ is isomorphic to $\B$, or
$\A$ is isomorphic to $\widehat{B}$
is a dichotomy, we proved in \cite{MatSogabe} that
the homotopy groups $\pi_i(\Aut(\OA)), i=1,2$ 
are complete invariants of the isomorphism class of 
a simple Cuntz--Krieger algebra $\OA$.
It was also proved in \cite{MatSogabe} that 
two extension groups 
$\Exts(\OA), \Extw(\OA)$ are complete set of invariants of 
simple Cuntz--Krieger algebra $\OA$ by using reciprocality.  
In the subsequent paper \cite{MatSogabe2},
by using the reciprocality,
the above results are generalized to more general
Kirchberg algebras with  finitely generated $\K$-groups.
These studies show that reciprocal duality plays an important role in studying the structure theory 
of Kirchberg algebras with finitely generated $\K$-groups.
We have however not known how to construct and realize the reciprocal algebra $\whatA$
from the original algebra $\A$,
because there is no concrete machinery to realize the Spanier--Whaitehead $\K$-dual $D(\A)$
from $\A$.

In the present paper, 
we will first describe a systematic construction of the reciprocal algebra $\whatA$
for a Kirchberg algebra $\A$ with finitely generated $\K$-groups.
In particular, we may concretely construct and realize the reciprocal algebra 
$\widehat{\O}_A$ for simple Cuntz--Krieger algebras $\OA$
in several different  ways. 
As in \cite{MatSogabe2}, the value 
$\chi(\A) = \rank(\K_0(\A)) - \rank(\K_1(\A))$
for any simple Cuntz--Krieger algebra $\A = \OA$ is always zero,  
wheras $\chi(\A) $ for $\A = \whatOA$ is always one, 
so that $\whatOA$ can not be realized as a Cuntz--Krieger algebra for any finite matrix.
Typical example of the reciprocal algebra is the reciprocal dual
$\widehat{\O}_2$ of the Cuntz algebra $\O_2$, 
which is realized as the Cuntz algebra $\O_\infty$.

The first strong (resp. weak) extension group $\Exts^1(\A)$ 
(resp. $\Extw^1(\A)$)
for a Kirchberg algebra is identified with the strong (resp. weak)
equivalence classes  $\Exts(\A)$ (resp. $\Extw(\A)$)
of essential unital extensions of $\A$ 
by the $C^*$-algebra $\calK$ 
of compact operators. 
For $m \in \Z$, take a unitary $u_m$ in the Calkin algebra 
$\calQ(H)$
of Fredholm index $m$.
It is well-known that for the extension $\tau_m:= \Ad(u_m)\circ \tau: \A \rightarrow \calQ(H)$  with the trivial 
extension $\tau:\A\rightarrow \B(H)$, the map 
$\iota_\A: \Z\ni m  \rightarrow [\tau_m]_s \in \Exts(\A)$ gives rise to 
an isomorphism
$\Extw(\A) \cong \Exts(\A)/\Z[\iota_\A(1)]_s.$
We will first prove the following theorem.

\begin{theorem}[{Theorem \ref{thm:construction}}] \label{thm:main1}
Let $\A$ be a Kirchberg algebra with finitely generated $\K$-groups.
For any essential unital extension 
$0 \rightarrow \calK
\rightarrow \E
\rightarrow \A
\rightarrow 0
$
of $\A$ and $\epsilon \in \{-1,1\}$, 
we have a systematic construction of a Kirchberg algebra $\O_\E$
with finitely generated $\K$-groups and a projection $e_\A \in \O_\E$ such that 
\begin{equation*}
(\K_0(\O_\E), [1_{\O_\E}]_0, [e_\A]_0, \K_1(\O_\E)) 
\cong 
(\Exts^1(\A), \epsilon [\E]_s, [\iota_\A(1)]_s, \Exts^0(\A))
\end{equation*}
where $[\E]_s\in \Exts^1(\A)$ is the strong equivalence class of 
$\E$ in $\Exts^1(\A)$.
Hence the isomorphism class of the Kirchberg algebra $\O_\E$ 
depends only on the strong equivalence class $[\E]_s$ of $\E$ in $\Exts^1(\A)$,
and the reciprocal dual $\whatA$ of $\A$ is realized as 
a corner of $\O_\E$, that is, 
\begin{equation*}
e_\A \O_\E e_\A \cong \widehat{\A} \quad \text{the reciprocal dual of } \A.
\end{equation*}
\end{theorem}
By using the construction given in Theorem \ref{thm:main1},
we may concretely  construct and realize the reciprocal dual $\whatOA$ 
of a simple Cuntz--Krieger algebra $\OA$
by using the Toeplitz extension 
$0 \rightarrow \calK \rightarrow \TA \rightarrow \OA \rightarrow 0$
(Theorem \ref{thm:mr} (ii)).
We prove the following theorem in Section \ref{sect:reciprocalCK}.
\begin{theorem}[{Theorem \ref{thm:mr}, Corollary \ref{cor:whatOA}, Proposition \ref{prop:ExelLaca} and Theorem \ref{thm:freeproduct}}]
\label{thm:main2}
 Let $A^t$ be the transposed matrix of 
 an $N \times N$ irreducible non-permutation matrix $A=[A(i,j)]_{i,j=1}^N$ 
 with entries in 
$\{0,1\}$.
  Then we may construct a Kirchberg algebra
  $\OATI$ and a projection $e_A$ in $\OATI$ in a concrete way such that 
the corner  $e_A \OATI e_A$ of $\OATI$ 
is isomorphic to the reciprocal algebra 
  $\whatOA$
  of the Cuntz--Krieger algebra $\OA$.
 The algebra
 $\OATI$ is constructed in the following three different ways:
 \begin{equation*}
 \begin{cases}
 &  \bullet \, \, \text{a Cuntz--Pimsner algebra for the Toeplitz algebra }\TAT, \\  
 &  \bullet \, \,\text{a universal } C^*-\text{algebra as a Exel--Laca algebra},  \\
 &  \bullet \, \,\text{a reduced free product } C^*-\text{algebra }\TAT*\OI.  
 \end{cases}
 \end{equation*}
\end{theorem}
In Section \ref{sect:generatorsrelations},
we present our first main result.
\begin{theorem}[{Theorem \ref{thm:univrelation}}]
Let $A =[A(i,j)]_{i,j=1}^N$ be an irreducible non-permutation matrix
with entries in $\{0,1\}$.
\begin{enumerate}
\renewcommand{\theenumi}{(\roman{enumi})}
\renewcommand{\labelenumi}{\textup{\theenumi}}
\item
The reciprocal dual $C^*$-algebra $\whatOA$ of the simple 
Cuntz--Krieger algebra $\OA$ is the universal $C^*$-algebra 
$C^*(r_i, s \mid i \in \N)$   
generated by a family $r_i, s, i \in \N$
of partial isometries subject to the following relation:
{\begin{enumerate}
\renewcommand{\theenumi}{(\arabic{enumi})}
\renewcommand{\labelenumi}{\textup{\theenumi}}
\item
$\sum_{j=1}^m r_j r_j^* < r_{N+1}^* r_{N+1} =r_{N+2}^* r_{N+2} =\cdots $ 
for all $m \in \N$,
\item
$r_i^* r_i 
= \sum_{j=1}^N A(j,i) r_j r_j^* + r_{N+1}^* r_{N+1} -\sum_{j=1}^N r_j r_j^*$ 
for $1 \le i \le N$,
\item
$s s^* = r_{N+1}^* r_{N+1} -\sum_{j=1}^N r_j r_j^*, \quad s^* s =1, $ 
\item $s^* r_{N+1} = r_{N+1}^* r_{N+1}.$ 
\end{enumerate}
}
\item
The reciprocal dual $C^*$-algebra $\whatOA$ 
is a unital simple purely infinite Exel--Laca algebra $\mathcal{O}_{\widehat{A}_\infty}$
defined by the  matrix $\widehat{A}_\infty$ where 
\begin{equation*}
\widehat{A}_\infty=
\left[
\begin{array}{ccccccc}
&                &1      &0     &0     &\cdots \\
&\text{\Huge A}^t&\vdots&\vdots&\vdots&\vdots \\
&                &1      &0     &0     &\cdots \\
1&\dots          &1      &1     &1     &\cdots \\
1&\dots          &1      &0     &0     &\cdots \\
1&\dots          &1      &0     &0     &\cdots \\
\vdots &\dots    &\vdots &\vdots&\vdots&\cdots \\
\end{array}
\right]
\end{equation*}
\end{enumerate}
\end{theorem}

The gauge action of the Cuntz--Krieger algebras $\OA$ 
is one of the most important circle action of $\OA$ to analyze the structure of
$\OA$. 
For the above realization of $\whatOA$ as a universal $C^*$-algebra 
generated by a family of partial isometries, 
there exists a natural gauge action 
on $\whatOA$ written $\hat{\gamma}^{A}$.
In Section \ref{sect:gaugeaction}, we will show that 
there exists a unique ground state for gauge action on $\whatOA$ whereas 
there is no KMS state 
(Proposition \ref{prop:groundstate2} and Remark \ref{remark:absenseKMSstate}).

As in the authors' previous paper \cite{MatSogabe},
the homotopy groups $\pi_i(\Aut(\OA))$ of the automorphism group $\Aut(\A)$ of $\A$
completely determine the algebraic structure of $\OA$.
One of the remarkable properties of the reciprocal dual of Cuntz--Krieger algebras is 
the fact that the homotopy groups  $\pi_i(\Aut(\OA))$ and
$\pi_i(\Aut(\whatOA))$ of their automorphism groups are isomorphic as abelian groups. 
The class $[\gamma^A]$ of the gage action $\gamma^A$ 
on $\OA$ occupies an interesting position in $\pi_1(\Aut(\OA))$
(cf. \cite{Cuntz1984}, \cite{KP}).
We have gauge action $\hat{\gamma}^{A}$ on the reciprocal dual $\whatOA$
which determines an element in $\pi_1(\Aut(\whatOA))$.

In Section \ref{gau}, we will prove the following our second main result.
\begin{theorem}[{Theorem \ref{thm:gaugeinpi1whatOA}}]
There exists an isomorphism
$\Phi: \pi_1(\Aut(\OA)) \rightarrow \pi_1(\Aut(\whatOA))$
of their fundamental groups satisfying
$\Phi([\gamma^A]) =[\hat{\gamma}^A]$.
\end{theorem}

In Section \ref{sec:Examples}, 
we present several examples of the reciprocal algebras $\whatA$
for Kirchberg algebras $\A$ with finitely generated $\K$-groups. 

In what follows, a Kirchberg algebra means a separable unital purely infinite simple $C^*$-algebra
satisfying UCT. 
An $N\times N$ matrix $A=[A(i,j)]_{i,j=1}^N$
is always assumed to be irreducible non-permutation with entries in $\{0,1\}.$ 
\subsection{Preliminaries}

Let $H$ be a separable infinite dimensional Hilbert sapce.
Let $\B(H)$ denote the $C^*$-algebra of bounded linear operators
on $H$, and $\calK(=\calK(H))$ the $C^*$-subalgebra of $\B(H)$ consisting of 
compact operators on $H$. 
Let $\A$ be a separable unital UCT $C^*$-algebra.
We mean by an extension of $\A$  a short exact sequence
\begin{equation}\label{eq:extEA}
0 \longrightarrow \calK
\longrightarrow \E
\longrightarrow \A
\longrightarrow 0
\end{equation}
of $C^*$-algebras. 
It is often written as $(\E)$ for short.
If $\E$ is unital and $\calK$ is an essential ideal of $\E$,
then the extension \eqref{eq:extEA} 
is said to be a unital essential extension of $\A$.
In what follows, an extension means  a unital essential extension.
It is well-known that  
 there is a bijective correspondence between unital essential extensions and
unital injective $*$-homomorphisms
$\tau: \A \rightarrow \calQ(H)$
from $\A$ to the Calkin algebra $\calQ(H) = \B(H)/ \calK$ (cf. \cite{Blackadar}). 
The unital injective $*$-homomorphism
$\tau: \A \rightarrow \calQ(H)$
is called a Busby invariant.
We often identify an extension and its Busby invariant.
Two Busby invariants $\tau_1, \tau_2:\A \rightarrow \calQ(H)$ 
are said to be strongly (resp. weakly) equivalent if there exists 
a unitary $U\in \B(H)$ (resp. $u\in \calQ(H)$)
such that $\tau_2(a) = \pi(U) \tau_1(a) \pi(U)^*$ 
(resp. $\tau_2(a) = u \tau_1(a) u^*$)
for all $a \in \A$, where
$\pi:\B(H) \rightarrow \calQ(H)$ is the natural quotient map.  
We fix an isomorphism $H\oplus H \cong H$, which induces 
an embedding $\calQ(H) \oplus \calQ(H) \hookrightarrow \calQ(H)$. 
Let $\Exts(\A)$ (resp. $\Extw(\A)$) 
be the set of strongly (resp. weakly) equivalence classes of Busby invariants
from $\A$ to $\calQ(H)$.
The embedding $\calQ(H) \oplus \calQ(H) \hookrightarrow \calQ(H)$
makes both $\Exts(\A)$ and  $\Extw(\A)$ abelian semi-groups.
It is well-known that if the $C^*$-algebra $\A$ is nuclear,
they become abelian groups (cf. \cite{Arveson}, \cite{Blackadar}, \cite{CE}).

Let $\A$ be a separable unital nuclear $C^*$-algebra.
The mapping cone $C_\A$ is defined by the $C^*$-algebra
$C_\A =\{ f \in C_0(0,1]\otimes \A \mid f(1) \in \mathbb{C} 1_\A\}$.
Define the four kinds of extension groups of $\A$
by Kasparov $\sqK$-groups:
\begin{equation}\label{eq:ExtKK}
\Exts^i(\A) := \sqK^{i+1}(C_\A, \mathbb{C}), \qquad
\Extw^i(\A) := \sqK^{i}(\A, \mathbb{C}),
\qquad i=0,1.
\end{equation} 
The latter groups 
$\sqK^{i}(\A, \mathbb{C}), i=0,1$ are also written as $\K^i(\A)$ and called the $\K$-homology groups for $\A$.
As in \cite{Blackadar}, \cite{Kasparov81}, \cite{Skandalis},
(cf.  \cite{HR}, \cite{PennigSogabe}), 
we know that 
\begin{equation*}
\Exts(\A) \cong \Exts^1(\A), \qquad 
\Extw(\A) \cong \Extw^1(\A).
\end{equation*}
There is a natural short exact sequence:
\begin{equation}\label{eq:exactCA}
0 
\longrightarrow S\A
\longrightarrow C_\A
\longrightarrow \mathbb{C}
\longrightarrow 0,
\end{equation}
where $S\A= C_0(0,1) \otimes \A$.
By applying the functor $\sqK(-,\mathbb{C})$
to \eqref{eq:exactCA}, we have a cyclic six term exact sequence:
\begin{equation} \label{eq:sixtermExt}
\begin{CD}
\Exts^0(\A) @>>> \Extw^0(\A) @>>> \Z \\
@AAA @. @VVV   \\
0 @<<< \Extw^1(\A) @<<< \Exts^1(\A).
\end{CD} 
\end{equation}
Then the  right vertical arrow in \eqref{eq:sixtermExt} 
is given by the homomorphism
$\iota_\A:\Z\rightarrow \Exts(\A)$
under the identification $\Exts^1(\A) = \Exts(\A)$.
Let us denote by $[1_\A]_0$ the class of the unit $1_\A$ of $\A$ in $\K_0(\A)$.

\section{Reciprocal duality and $\K$-theoretic duality for extensions}
\subsection{Reciprocal duality}
In this subsection, 
we summarize basic facts concerning on reciprocal duality
studied in \cite{Sogabe2022}, \cite{PennigSogabe}, \cite{MatSogabe} and \cite{MatSogabe2}.
Let $\A$ be a separable $C^*$-algebra with finitely generated $\K$-groups.
The Spanier--Whitehead $\K$-dual of $\A$
 is defined to be a separable $C^*$-algebra  $D(\A)$ such that 
 $\sqK(D(\A), \mathbb{C}) \cong \sqK(\mathbb{C}, \A)$ \cite{KS} (cf. \cite{KP}).
Although the isomorphism class of $D(\A)$
 is not uniquely determined by $\A$,
 it is unique up to $\sqK$-equivalence.  
Let us recall the definition of reciprocality introduced in \cite{Sogabe2022}.
\begin{definition}[{\cite{Sogabe2022}}]
Let $\A$ be a separable $C^*$-algebra with finitely generated $\K$-groups.
Then a separable $C^*$-algebra $\B$ with finitely generated $\K$-groups
is said to be {\it reciprocal}\/ to $\A$ if  
$\A \sim_{\sqK} D(C_\B)$
and
$\B \sim_{\sqK} D(C_\A).$
\end{definition}
It is proved in \cite{Sogabe2022} that for a Kirchberg algebra $\A$
with finitely generated $\K$-groups, there exists   
 a Kirchberg algebra $\B$ with finitely generated $\K$-groups
 which is reciprocal to $\A$. 
 It  is unique up to isomorphism, 
 and  written as $\whatA$.
 By definition, it is obvious that 
 $\B$ is reciprocal to $\A$ if and only if 
 $\A$ is reciprocal to $\B$, so that 
 $\widehat{\whatA} \cong \A$. 
  
Let $G, H$ be abelian groups with $g \in G, h \in H$.
We write $(G, g) \cong (H,h)$ if there exists an isomorphism 
$\Phi:G\rightarrow H$ such that $\Phi(g) = h.$
The following proposition is seen in \cite[Proposition 3.8]{MatSogabe2}
without proof.
We give its proof for the sake of completeness.
\begin{proposition}\label{prop:basic2}
Let $\A, \B $ be Kirchberg algebras with finitely generated $\K$-groups.
Then the following four conditions are equivalent:
\begin{enumerate}
\renewcommand{\theenumi}{(\roman{enumi})}
\renewcommand{\labelenumi}{\textup{\theenumi}}
\item $\whatA \cong \B.$
\item $\K_i(\A) \cong \Exts^{i+1}(\B)$ and $\Exts^i(\A) \cong \K_{i+1}(\B), \quad i=0,1.$
\item $(\Exts^1(\A), \iota_\A(1)) \cong (\K_0(\B), [1_\B]_0)$ and $\Exts^0(\A) \cong \K_1(\B)$.
\item $(\K_0(\A), [1_\A]_0) \cong (\Exts^1(\B), \iota_\B(1))$ and $\K_1(\A) \cong \Exts^0(\B)$.
\end{enumerate}
\end{proposition}
\begin{proof}
We note the following:
\begin{equation}
\K_i(\A) 
=  \K_i(D(C_{\whatA})) 
  \cong
 \sqK^i(C_{\whatA}, \mathbb{C})
  =\Exts^{i+1}(\whatA), \quad i=0,1.\label{eq:2.1} 
\end{equation}
 
(i) $\Longleftrightarrow$ (ii):
Assume (i). The assertion (ii) immediately follows from 
\eqref{eq:2.1} and the fact that $\widehat{\widehat{\A}} \cong \A$.
Assume (ii). 
By \eqref{eq:2.1}, 
we have
$\Exts^{i+1}(\B) \cong \K_i(D(C_\B))$
so that the condition 
$\K_i(\A) \cong \Exts^{i+1}(\B)$
implies $D(C_\B)\sim_{\sqK}\A$.
Symmetrically the condition
$\Exts^i(\A) \cong \K_{i+1}(\B)$
implies $D(C_\A)\sim_{\sqK}\B$,
showing that $\B$ is reciprocal to $\A$.

(i) $\Longleftrightarrow$ (iii):
Assume (i) and hence $\A \cong \whatB.$
By \eqref{eq:2.1} for $\B$, we have
$\Exts^1(\A) \cong \Exts^1(\whatB) \cong  \K_0(\B)$
and
\begin{equation*}
\Exts^1(\A)/\Z\iota_\A(1) 
\cong \Extw^1(\A) =\K^1(\A)
\cong\K_1(D(\A)) \cong \K_1(C_\whatA) 
\cong \K_0(\A)/\Z[1_\A]_0. 
\end{equation*}
By \cite[Corollary 2.20]{Sogabe2022},
we have $(\Exts^1(\A), \iota_\A(1)) \cong (\K_0(\B), [1_\B]_0)$.
The condition
$\Exts^0(\A) \cong \K_1(\B)$ 
follows from \eqref{eq:2.1} for $\whatA \cong \B.$

Assume (iii).
As 
$\Exts^1(\A)/\Z\iota_\A(1) \cong \Extw^1(\A)$
and
$\K_0(\B)/\Z[1_\B]_0 \cong \Extw^1(\B)$,
by \eqref{eq:2.1}, we have 
$\Extw^1(\A) \cong \Extw^1(\whatB).$
Since
$\K_i(\B) \cong \Exts^{i+1}(\whatB), i=0,1,$
we have
$\Exts^{i+1}(\A) \cong \Exts^{i+1}(\whatB),$
so that $\A \cong\whatB$  
by \cite[Theorem 4.3]{MatSogabe2},
and hence $\whatA \cong \B$.

(i) $\Longleftrightarrow$ (iv):
Since $\whatA \cong \B$ is equivalent to $\A \cong \whatB$,
the equivalence (i) $\Longleftrightarrow$ (iv)
follows from the equivalence 
(i) $\Longleftrightarrow$ (iii).
\end{proof}
\begin{corollary}\label{cor:whatA}
Let $\A $ be a Kirchberg algebra with finitely generated $\K$-groups.
Then the reciprocal dual $\whatA$ is characterized to be 
a Kirchberg algerba with finitely generated $\K$-groups satisfying 
\begin{equation*}
(\K_0(\whatA), [1_{\whatA}]_0, \K_1(\whatA)) \cong
(\Exts^1(\A), \iota_\A(1), \Exts^0(\A)).
\end{equation*}
\end{corollary}
The two invariants introduced in \cite{MatSogabe2}  
\begin{align}
\chi(\A):= & \rank(\K_0(\A) ) -\rank(\K_1(\A)) \in \Z, \label{eq:chi(A)}\\
w(\A) := & \rank(\K_0(\A)) - \rank(\K_0(\A)/\Z[1_\A]_0) \in \{0, 1\} \label{eq:w(A)}
\end{align}
for a Kirchberg algebra $\A$ 
 with finitely generated $\K$-groups give a hierarchy of the Kirchberg algebras 
 with finitely generated $\K$-groups,
 where $\rank(G)$ for a finitely generated abelian group $G$ is defined by 
 $\dim_\mathbb{Q}(G\otimes_\Z\mathbb{Q}).$
The values make a clear difference between $\whatA$ and $\A$ as in the following lemma. 
\begin{lemma}[{\cite[Lemma 3.7]{MatSogabe2}}]\label{lem:chiw}
Let $\A$ be a Kirchberg algebra  with finitely generated $\K$-groups.
Then we have
$$
\chi(\A) + \chi(\whatA) =1, \qquad  w(\A) + w(\whatA) =1.
$$
\end{lemma}
A typical example of the reciprocal Kirchberg algebras are the Cuntz algebras $\O_2, \OI$ with $\widehat{\O_2}\cong \OI$.

One of the main purpose of the present paper is to present and realize
the reciprocal dual $\whatA$ of a Kirchberg algebra $\A$,
especially of a Cuntz--Krieger algebra $\OA$ in a systemtic and concrete way.

\subsection{$\K$-theoretic duality for extensions}

There is another $\K$-theoretic duality in $C^*$-algebra contexts,
that is a $\K$-theoretic duality for extensions introduced in 
\cite{MaJMAA2024}(cf. \cite{PennigSogabe}).
$\K$-theoretic duality for extensions of $C^*$-algebras is useful in construction
of reciprocal algebras in our further discussions. 
\begin{definition}[{\cite{MaJMAA2024}, \cite{PennigSogabe}}]
Let $\A, \B$ be separable unital nuclear $C^*$-algebras.
Two extensions
$$
(\E): \quad 0 \rightarrow \calK \rightarrow \E \rightarrow \A \rightarrow 0
\quad \text{ and }
\quad
(\F): \quad 0 \rightarrow \calK \rightarrow \F \rightarrow \B \rightarrow 0
$$
are called a strong $\K$-theoretic duality pair  with respect to $\epsilon\in \{-1,1\}$
if the following two conditions hold:
\begin{enumerate}
\renewcommand{\theenumi}{(\arabic{enumi})}
\renewcommand{\labelenumi}{\textup{\theenumi}}
\item
There exist
isomorphisms
$\Phi_\A^i: \Exts^i (\A)\rightarrow \K_{i+1}(\F), i=0,1$
such that 
\begin{equation}\label{eq:SKdualA}
\Phi_\A^1([\E]_s) = \epsilon [1_\F]_0, \qquad
\Phi_\A^1([\iota_\A(1)]_s) = [e]_0,
\end{equation}
where $[e]_0 \in \K_0(\F)$ is the class $[e]_0$ of a projection $e \in \calK$
of rank one.
\item
There exist
isomorphisms
$\Phi_\B^i: \Exts^i (\B)\rightarrow \K_{i+1}(\E), i=0,1$
such that 
\begin{equation}\label{eq:SKdualB}
\Phi_\B^1([\F]_s) = \epsilon [1_\E]_0, \qquad
\Phi_\B^1([\iota_\B(1)]_s) = [e]_0,
\end{equation}
where $[e]_0 \in \K_0(\E)$ is the class $[e]_0$ of a projection $e \in \calK$
of rank one.
\end{enumerate}
\end{definition} 
We note that 
the condition \eqref{eq:SKdualA} (resp. \eqref{eq:SKdualB})
automatically implies that
$\Extw^1(\A)\cong \K_0(\B)$ and $\Extw^0(\A)\cong \K_1(\B)$
(resp.
$\Extw^1(\B)\cong \K_0(\A)$ and $\Extw^0(\B)\cong \K_1(\A)).$
The following result was proved in \cite[Theorem 5.7]{PennigSogabe}.
\begin{proposition}[{\cite[Theorem 5.7]{PennigSogabe}}] \label{prop:PennigSogabe}
Let $(\E): 0 \rightarrow \calK \rightarrow \E \rightarrow \A \rightarrow 0$
be an extension of a separable unital nuclear $C^*$-algebra $\A$ with finitely generated $\K$-groups.
For $\epsilon \in \{-1,1\}$, 
there exists a separable unital nuclear $C^*$-algebra $\B$
with finitely generated $\K$-groups
and an extension  
$(\F): 0 \rightarrow \calK \rightarrow \F \rightarrow \B \rightarrow 0$
such that $(\E)$ and 
$(\F)$ is a stong $\K$-theoretic duality pair with respect to $\epsilon$. 
If in particular $\A$ is a Kirchberg algebra, 
then one may take $\B$ as a Kirchberg algebra.
\end{proposition}
A typical example of a strong $\K$-theoretic duality pair is the Toeplitz  
extensions  of Cuntz--Krieger algebras:
\begin{equation}\label{eq:TATAT}
(\TA): \, \, 0 \rightarrow \calK \rightarrow \TA \rightarrow \OA \rightarrow 0, 
\qquad
(\TAT): \, \, 0 \rightarrow \calK  \rightarrow \TAT \rightarrow \OAT \rightarrow 0, 
\end{equation}
which is a strong $\K$-theoretic duality pair with respect to $\epsilon = -1$
(\cite{MaJMAA2024}).


\section{General construction of the reciprocal dual}
For abelian groups $G_i, H_i, i=0,1$ with elements 
$g_0 \in G_0$ and $h_0 \in H_0$, we write
$(G_0, g_0, G_1) \cong (H_0, h_0, H_1)$
if there exist isomorphisms $\Phi_i: G_i \rightarrow H_i, \, i=0,1$
of groups such that $\Phi_0(g_0) = h_0.$
We similarly write  
$(G_0, g_0, g'_0, G_1) \cong (H_0, h_0,h'_0, H_1)$
with elements $g_0, g'_0 \in G_0$ and  $h_0, h'_0 \in H_0$
if there exist isomorphisms $\Phi_i: G_i \rightarrow H_i, \, i=0,1$
of groups such that $\Phi_0(g_0) = h_0, \Phi_0(g'_0) = h'_0.$

We will show the following theorem which gives us a systematic construction of the reciprocal dual
$\whatA$ from $\A$, by using Cuntz--Pimsner--Kumjian 's construction for an extension of $\A$
(\cite{Pimsner}, \cite{Kumjian}). 
The systematic construction will give us a concrete realization of the reciprocal dual
$\whatOA$ of a simple Cuntz--Krieger algebr $\OA$ in Theorem \ref{thm:mr} (ii)
 and Corollary \ref{cor:whatOA}.
\begin{theorem}\label{thm:construction}
Let $\A$ be a Kirchberg algebra with finitely generated $\K$-groups.
For any essential unital extension 
$0 \longrightarrow \calK
\longrightarrow \E
\longrightarrow \A
\longrightarrow 0
$
of $\A$ and $\epsilon \in \{-1,1\}$, 
we have a systematic construction of a Kirchberg algebra $\O_\E$
with finitely generated $\K$-groups and a projection $e_\A \in \O_\E$ such that 
\begin{equation*}
(\K_0(\O_\E), [1_{\O_\E}]_0, [e_\A]_0, \K_1(\O_\E)) 
\cong 
(\Exts^1(\A), \epsilon [\E]_s, [\iota_\A(1)]_s, \Exts^0(\A))
\end{equation*}
where $[\E]_s\in \Exts^1(\A)$ is the strong equivalence class of 
$\E$ in $\Exts^1(\A)$.
Hence the isomorphism class of the Kirchberg algebra $\O_\E$ 
depends only on the strong equivalence class $[\E]_s$ of $\E$ in $\Exts^1(\A)$,
and the reciprocal dual $\whatA$ of $\A$ is realized as 
a corner of $\O_\E$, that is, 
\begin{equation*}
e_\A \O_\E e_\A \cong \whatA \quad \text{the reciprocal dual of } \A.
\end{equation*}
\end{theorem}
\begin{proof}
 For an extension
$(\E): 
0 \longrightarrow \calK
\longrightarrow \E
\longrightarrow \A
\longrightarrow 0
$
of $\A$ and $\epsilon \in \{1,-1 \}$, 
by \cite{PS}, 
there exists an extension
$(\F):
0 \longrightarrow \calK
\longrightarrow \F
\longrightarrow \B
\longrightarrow 0
$ of a Kirchberg algebra $\B$ such that 
they are strong $\K$-theoretic duality pair with the isomorphisms
 \begin{align*}
\Phi_A: & (\Exts^1(\A), [\E]_s, [\iota_\A(1)]_s, \Exts^0(\A))
\rightarrow (\K_0(\F), \epsilon [1_\F]_0, [e]_0, \K_1(\F)),\\
\Phi_B: & (\Exts^1(\B), [\F]_s, [\iota_\B(1)]_s, \Exts^0(\B))
\rightarrow (\K_0(\E), \epsilon [1_\E]_0, [e]_0, \K_1(\E)). 
\end{align*}
Hence there are isomorphisms
$\Phi_A^1: \Exts^1(\A) \rightarrow \K_0(\F),$
$\Phi_B^1: \Exts^1(\B) \rightarrow \K_0(\E) $
satisfying \eqref{eq:SKdualA}, \eqref{eq:SKdualB}, respectively.
For the extension
$(\F)$,
there exists a 
$*$-homomorphism $\pi_\F: \F \rightarrow \B(H)$ 
for the Busby invariant $\tau_\F: \B \rightarrow \calQ(H)$
such that the diagram 
\begin{equation*}
\begin{CD}
0 @>>> \calK @>>> \F @>>> \B @>>> 0 \\
@. @|  @V{\pi_\F}VV @VV{\tau_\F}V @. \\ 
0 @>>> \calK @>>> \B(H) @>>> \calQ(H) @>>> 0 \\
\end{CD}
\end{equation*}
commutes.
Since $\calK$ is essential in $\F$, 
the $*$-homomorphism $\pi_\F: \F \rightarrow \B(H)$ is injective.
Following the A. Kumjian's construction in \cite{Kumjian}(see also \cite{Pimsner})
of Cuntz--Toeplitz--Pimsner algebras, 
we will get 
 our desired  $C^*$-algebras 
$\O_{\E}$ and the 
reciprocal algebra $\widehat{\A}$ in the following way.
We first represent $\F$ on the Hilbert space $H$
by the map
$\pi_\F: \F \rightarrow \B(H)$ 
and consider the Hilbert $C^*$-bimodule 
$H_\F:=(H\otimes_\mathbb{C} l^2(\mathbb{N}))\otimes_\mathbb{C} \F$ 
over $\F$
where 
the $\F$-valued inner product on $H_\F$ is defined by 
$<\xi\otimes x, \zeta\otimes y>_\F:=<\xi, \zeta>_\mathbb{C}x^*y\in \F$ 
for 
$\xi, \zeta\in H\otimes l^2(\mathbb{N})$, $x, y\in \F$,
and
the left action 
$\varphi_\F: \F \rightarrow \mathcal{L}(H\otimes_\mathbb{C} l^2(\mathbb{N})\otimes_\mathbb{C} \F)$  
of $\F$ to the adjointable bounded module maps on $H_\F$    
is given by $\varphi_\F := \pi_\F \otimes 1 \otimes 1: 
\F\otimes 1\otimes 1 \subset \mathcal{L}(H\otimes l^2(\mathbb{N})\otimes \F).$ 
We second take Pimsner's tensor algebra 
$\mathcal{T}_{H_\F}$ 
which is KK-equivalent to $\F$ via the inclusion 
$\F\hookrightarrow \mathcal{T}_{H_\F}$.
Since $\F\otimes 1\otimes 1\cap \mathcal{K}(H_\F)=0$, 
Kumjian's results 
\cite[Proposition 2.1]{Kumjian} and \cite[Theorem 3.1]{Kumjian}  
together with \cite[Corollary 4.5]{Pimsner}
tell us that 
the $C^*$-algebra $\O_{H_\F}$ 
is simple purely infinite and hence 
it is a Kirchberg algebra such that
$\O_{H_\F}$ and $\T_{H_\F}$ are canonically isomorphic, 
and 
the natural embedding 
$\F \hookrightarrow \T_{H_\F}$
yields a $\sqK$-equivalence.
Let $e_\A \in \O_{H_\F}$ be the projection defined by 
the embeddings
$
e\in \calK \hookrightarrow \F \hookrightarrow \T_{H_\F} \cong \O_{H_\F}.
$
We then have
\begin{align*}
(\K_0(\O_{H_\F}), [1_{\O_{H_\F}}]_0, [e_\A]_0, \K_1(\O_{H_\F})) 
= &  
(\K_0(\T_{H_\F}), [1_{\T_{H_\F}}]_0, [e]_0,  \K_1(\T_{H_\F})) \\
= &  
(\K_0(\F), [1_{\F}]_0, [e]_0, \K_1(\F)) \\
\cong &  
(\Exts^1(\A), \epsilon [\E]_s, [\iota_\A(1)]_s, \Exts^0(\A)).
\end{align*}
Hence there exists an isomorphism 
$\Phi_0: \K_0(\O_{H_\F}) \rightarrow \Exts^1(\A)$ such that 
$\Phi_0( [1_{\O_{H_\F}}]_0) = \epsilon [\E]_s$.
For $\epsilon = -1$,
by composing the automorphism 
$g \in \K_0(\O_{H_\F}) \rightarrow -g \in \K_0(\O_{H_\F}),$
one has 
\begin{equation*}
(\K_0(\O_{H_\F}), [1_{\O_{H_\F}}]_0, \K_1(\O_{H_\F})) 
\cong   
(\Exts^1(\A), [\E]_s, \Exts^0(\A)),
\end{equation*}
so the Kirchberg--Phillips classification theorem shows that
the isomorphism class of the Kirchberg algebra $\O_{H_\F}$
for the $C^*$-algebra $\A$
depends only on the class $[\E]_s$ in $\Exts^1(\A)$ (i.e., independent of $\epsilon\in \{\pm 1\}$ and $\F$).
Hence we may write the $C^*$-algebra $\O_{H_\F}$
as $\O_\E$, which is the desired algebra. 
By Proposition \ref{prop:basic2} (i) $\Longrightarrow$ (iii), 
 we have an isomorphism
$$
(\Exts^1(\A), [\iota_\A(1)]_s, \Exts^0(\A))
\cong
(\K_0(\widehat{\A}), [1_{\widehat{\A}}]_0, \K_1(\widehat{\A})) 
$$
and hence 
$$
(\K_0(\O_{\E}), [e_\A]_0, \K_1(\O_{\E})) 
\cong
(\K_0(\widehat{\A}), [1_{\widehat{\A}}]_0, \K_1(\widehat{\A})),
$$ 
proving that 
$e_\A \O_{\E} e_\A \cong \widehat{\A}.$
\end{proof}
\begin{remark}
{\bf 1.}
In Theorem \ref{thm:construction},
if we take  the extension
$\E_1 : = \iota_\A(1)$ 
of $\A$, 
we have $[\E_1]_s = [\iota_\A(1)]_s$  and hence 
$
\O_{\E_1} \cong \whatA, 
$ 
because 
we have
\begin{align*}
(\K_0(\O_{\E_1}), [1_{\O_{\E_1}}]_0, \K_1(\O_{\E_1})) 
\cong   
(\Exts^1(\A),  [\E_1]_s, \Exts^0(\A)) 
\cong 
(\K_0(\widehat{\A}), [1_{\widehat{\A}}]_0, \K_1(\widehat{\A})),
\end{align*}
so  that 
$\O_{\E_1} \cong \whatA.$

{\bf 2.}
The pair $(\O_\E, e_\A)$ is compatible to strong $\K$-theoretic duality for extensions 
in the following sense.
Two extensions  
$ 
0 \rightarrow \calK
\rightarrow \E
\rightarrow \A
\rightarrow 0,  \text{ and }
0 \rightarrow \calK
\rightarrow \F
\rightarrow \B
\rightarrow 0
$ 
are strong $\K$-theoretic duality pair
with respect to $\epsilon \in \{1,-1 \}$ 
if and only if 
 \begin{align*}
 (\K_0(\O_\E), [1_{\O_\E}]_0, [e_\A]_0, \K_1(\O_\E))
\cong & (\K_0(\F), \epsilon [1_\F]_0, [e]_0, \K_1(\F)),\\
 (\K_0(\O_\F), [1_{\O_\F}]_0, [e_\B]_0, \K_1(\O_\F))
\cong & (\K_0(\E), \epsilon [1_\E]_0, [e]_0, \K_1(\E)). 
\end{align*}
\end{remark}

\section{The reciprocal dual of Cuntz--Krieger algebras}\label{sect:reciprocalCK}
In this section, we focus on studying concrete constructions and realizations
of the reciprocal dual $\whatOA$ of a simple Cuntz--Krieger algebra $\OA.$ 
Throughout this section, 
let $A = [A(i,j)]_{i,j=1}^N$ be an irreducible non-permutation matrix with entries in $\{0,1\}.$
It forces us the algebra $\OA$ to be a Kirchberg algebra.
The Cuntz--Krieger algebra $\OA$ 
is the universal C*-algebra generated by partial isometries 
$S_1, \dots, S_N$ subject to the relations:
$
1=\sum_{j=1}^NS_jS_j^*,\, S^*_iS_i=\sum_{j=1}^NA(i, j)S_jS_j^*, \, \,  i=1,\dots,N.
$
By Corollary \ref{cor:whatA} together with \cite[Theorem 3.3 (ii)]{MaAnalMath2024}
and \cite[Lemma 4.2]{MaJMAA2024}, we know that
the reciprocal dual $C^*$-algebra $\whatOA$ of $\OA$ 
is characterized in terms of the $\K$-theory data in the following way.
\begin{proposition}\label{prop:chracterizationwhatOA}
Let $\widehat{A}$ be the $N\times N$ matrix $A + R_1 - AR_1$
where $R_1$ is the $N \times N$ matrix whose first row is the vector 
$[1,\dots,1]$ and the other rows are zero vectors.
Let $A_T$ be the $(N+1) \times N$ matrix defined by
$A_T = 
\begin{bmatrix} 
-1& \cdots &-1 \\
  & I - A &\\
\end{bmatrix}.
$
We then have
\begin{equation}\label{eq:KwhatOA}
\begin{split}
(\K_0(\whatOA), [1_{\whatOA}]_0) & \cong 
(\Z^N/(I - \widehat{A})\Z^N, 
[(I-A) 
\begin{bmatrix}
1\\
0\\
\vdots\\
0
\end{bmatrix}]),\\
\K_1(\whatOA) & \cong 
\Ker(A_T: \Z^N \rightarrow \Z^{N+1}). 
\end{split}
\end{equation}
\end{proposition}
In this section, we will concretely construct the Kirchberg algebra $\whatOA$ from $\OA$
satisfying \eqref{eq:KwhatOA} and realize it in several different ways.
We first present a difference between $\whatOA$ and $\OA$ via the invariant $\chi$ and $w$.

\subsection{The comparison between $\whatOA$ and  $\OA$} 

For a finitely generated abelian group $G$,
we write its torsion part and torsion free part as
$\Tor(G)$ and $\Free(G)$, respectively.
Let us recall the invariant $\chi(\A)$
defined in \eqref{eq:chi(A)}.
We first provide a lemma which characterizes Cuntz--Krieger algebras in Kirchberg algebras
by using $\K$-groups.
Since the proof is obvious from \cite{Ro}, so we omit it.
\begin{lemma}\label{lem:characterizationCK}
Let $\A$ be a Kirchberg algebra with finitely generated $\K$-groups.
Then the following three conditions are equivalent:
\begin{enumerate}
\renewcommand{\theenumi}{(\roman{enumi})}
\renewcommand{\labelenumi}{\textup{\theenumi}}
\item 
$\A \cong \OA$ for some simple Cuntz--Krieger algebra $\OA$.
\item $\Free(\K_0(\A)) \cong \K_1(\A).$
\item $\chi(\A) = 0$ and $\Tor(\K_1(\A)) \cong 0.$
\end{enumerate} 
\end{lemma}
The proposition below clarifies the difference 
between $\whatOA$ and $\OA$.
\begin{proposition}\label{prop:characterizationwhatOA}
Let $\B$ be a Kirchberg algebra with finitely generated $\K$-groups.
Then the following three conditions are equivalent:
\begin{enumerate}
\renewcommand{\theenumi}{(\roman{enumi})}
\renewcommand{\labelenumi}{\textup{\theenumi}}
\item 
$\B \cong  \widehat{\O}_A$ for some simple Cuntz--Krieger algebra $\OA$.
\item $\Free(\K_0(\B)) \cong \K_1(\B) \oplus \Z.$
\item $\chi(\B) = 1$ and $\Tor(\K_1(\B)) \cong 0.$
\end{enumerate} 
\end{proposition}
\begin{proof}
Since 
$\Exts^i(\widehat{\O}_A) \cong \K_{i+1}(\OA), i=0,1$ 
 by Lemma \ref{lem:characterizationCK},
the condition (i) is equivalent to the condition 
\begin{equation}\label{eq:equivalenceBOA}
\Free(\Exts^1(\B)) \cong  \Exts^0(\B).
\end{equation}
By the cyclic six term exact sequence \eqref{eq:sixtermExt} for $\B$,
 the condition \eqref{eq:equivalenceBOA}
 is equivalent to the condition
\begin{equation}\label{eq:equivalenceBOA2}
 \Extw^0(\B) \cong \Free(\Extw^1(\B)) \oplus \Z.
\end{equation}
By the Universal Coefficient Theorem (UCT) (cf. \cite{Blackadar}), 
we see that
$$
\Extw^0(\B) \cong \Free(\K_0(\B)) \oplus \Tor(\K_1(\B)),
\quad
\Extw^1(\B) \cong \Free(\K_1(\B)) \oplus \Tor(\K_0(\B)),
$$
showing that the equality \eqref{eq:equivalenceBOA2}
is equivalent to the condition (ii).
The equivalence between 
(ii) $\Longleftrightarrow$ (iii) is obvious.
\end{proof}
Let us recall the invariant $w(\A)$ defined in \eqref{eq:w(A)}.
The following proposition 
shows us more clear difference between $\OA$ and $\whatOA$ through the invariant $w(\A)$.
\begin{proposition}\label{prop:equivalentwhatOA}
Let $\B$ be a Kirchberg algebra with finitely generated $\K$-groups
such that $\B \cong \widehat{\O}_A$ for some  simple Cuntz--Krieger algebra $\OA$.
Then we have
\begin{enumerate}
\renewcommand{\theenumi}{(\roman{enumi})}
\renewcommand{\labelenumi}{\textup{\theenumi}}
\item 
If $w(\B) =1$, that is, $[1_\B]_0$ in $\K_0(\B)$ is non-torsion, 
then
\begin{align*}
\K_1(\OA) \cong \K_1(\B), \quad 
\K_0(\OA) \cong \K_0(\B)/\Z[1_\B]_0, \quad 
\Free(\K_0(\OA)) \oplus \Z \cong \Free(\K_0(\B)).
\end{align*}  
\item 
If $w(\B) =0$, that is,  $[1_\B]_0$  in $\K_0(\B)$ is torsion, 
then
\begin{align*}
\K_1(\OA) \cong \K_1(\B)\oplus \Z, \quad 
\K_0(\OA) \cong \K_0(\B)/\Z[1_\B]_0, \quad 
\Free(\K_0(\OA)) \cong \Free(\K_0(\B)).
\end{align*}  
\end{enumerate}
\end{proposition}
\begin{proof}
Assume that $\B = \widehat{\O}_A$ for some simple Cuntz--Krieger algebra $\OA$.
Since
\begin{equation*}
\K_i(C_\B) \cong \K^i(D(C_\B)) \cong \K^i(\OA) \cong \K_{i+1}(\OA), \quad i=0,1
\end{equation*}
together with the the standard cyclic six term sequence 
for the short exact sequence \eqref{eq:exactCA}
for $\A = \B$, we have
\begin{equation}
\begin{CD} \label{eq:6termK1B}
\K_1(\B) @>{}>> \K_1(\OA) @>{}>> \Z \\
@AAA  @. @VV{\partial_\B}V \\
0 @<<< \K_0(\OA) @<<< \K_0(\B)
\end{CD}
\end{equation}
Now the homomorphism
$\partial_\B: \Z\rightarrow \K_0(\B)$
satisfies $\partial_\B(1) = [1_\B]_0$ in $\K_0(\B).$
Since $\Ker(\partial_\B)\subset \Z$ is a subgroup of $\Z$,
one may find a nonnegative integer $n_\B$ such that 
$\Ker(\partial_\B) = n_B \Z$.
We then have two short exact sequences
\begin{align}
0 \longrightarrow  \K_1(\B) \longrightarrow & \K_1(\OA)
 \longrightarrow n_\B \Z \longrightarrow 0, \label{eq:partialB1} \\ 
0 \longrightarrow \partial_\B(\Z) \longrightarrow & \K_0(\B)
 \longrightarrow \K_0(\OA) \longrightarrow 0. \label{eq:partialB2}
\end{align}
By \eqref{eq:partialB2}, we immediately see that 
$\K_0(\OA) \cong \K_0(\B)/\Z[1_\B]_0.$

We have two cases.

(i) $[1_\B]_0$ in $\K_0(\B)$ is non-torsion:
This case is equivalent to the condition that $n_\B =0$.
It implies that 
$\partial_\B: \Z\rightarrow \K_0(\B)$
is injective, so that 
we have
$$
\K_1(\OA) \cong \K_1(\B), \quad 
\Free(\K_0(\OA)) \oplus \Z \cong \Free(\K_0(\B)).
$$

(ii) $[1_\B]_0$ in $\K_0(\B)$ is  torsion:
This case is equivalent to the condition that $n_\B \ne 0$.
By the exact sequence \eqref{eq:partialB1},
we have
$$
\K_1(\OA) \cong \K_1(\B)\oplus \Z, \quad 
\Free(\K_0(\OA)) \cong \Free(\K_0(\B)).
$$
\end{proof}


\subsection{The Cuntz--Krieger--Toeplitz extensions}

We will provide the Toeplitz algebra $\TA$, introduced in \cite{EFW79, EV},
 for a matrix $A$ 
with entries in $\{0,1\}$ to construct $\whatOA$ in a concrete way.
Let $\{e_i\}_{i=1}^N$ be an orthonormal basis of 
$\mathbb{C}^N$ and let $F_A$ 
be the subspace of the full Fock space 
$\bigoplus_{n=0}^\infty (\mathbb{C}^N)^{\otimes n}$ 
with a vacuum vector $\Omega_A$ 
(i.e., ($\mathbb{C}^N)^{\otimes 0}=\mathbb{C}\Omega_A$)
spanned by the vectors
\[
e_{i_1}\otimes e_{i_2}\otimes\cdots\otimes e_{i_n}\in (\mathbb{C}^N)^{\otimes n}
\quad \text{ with } \quad
A(i_k, i_{k+1})=1, \quad k=1,\dots, n-1.
\]
Let $e_A : F_A\to (\mathbb{C}^N)^{\otimes 0}=\mathbb{C}\Omega_A$ 
be the minimal projection onto the vacuum vector $\Omega_A$.
The creation operators 
$T_i : F_A\to F_A,\;i=1,\dots, N$ 
are defined by
\begin{equation}\label{eq:TionFA}
\begin{split}
T_i(e_{i_1}\otimes e_{i_2}\otimes\cdots\otimes e_{i_n})
& = 
A(i, i_1)e_i\otimes(e_{i_1}\otimes e_{i_2}\otimes\cdots\otimes e_{i_n}), \\
T_i(\Omega_A)
& = e_i
\end{split}
\end{equation}
which satisfy
\begin{equation}
1=\sum_{j=1}^NT_jT_j^* + e_A, 
\qquad 
T_i^*T_i= \sum_{j=1}^NA(i, j)T_jT_j^* +e_A, \quad i=1,\dots,N.
\label{CKT}
\end{equation}
The $C^*$-algebra generated by 
$e_A, T_1,\dots, T_N$ 
is called the Toeplitz algebra $\mathcal{T}_A$ (\cite{EFW79}, cf. \cite{EV}).
For words 
$\mu=(\mu_1,\dots,\mu_m),\,  \nu=(\nu_1, \dots, \nu_n)$,
the operators 
$T_\mu e_A T_\nu^*:=T_{\mu_1}\cdots T_{\mu_m}e_A T_{\nu_n}^*\cdots T_{\nu_1}^*$ 
generate an ideal of $\mathcal{T}_A$ 
 isomorphic to $\calK$ and one has the Cuntz--Krieger--Toeplitz extension
\[
0 \rightarrow \calK \rightarrow \TA \rightarrow  \OA \rightarrow 0.
\]
It is easy to see that 
 $\TA$ is the universal $C^*$-algebra generated by 
a projection $e_A$ and partial isometries 
$ T_1,\dots, T_N$ subject to the relations \eqref{CKT}.

\subsection{A concrete realization of $\whatOA$}\label{proa}
We assume that $A=[A(i,j)]_{i,j=1}^N$ 
is an $N\times N$ irreducible non-permutation matrix with entries in $\{0,1\}.$
Let  
$\mathcal{O}_{A^\infty}$ 
be the universal $C^*$-algebra generated by partial isometries 
$S_i, i \in \N$ subject to the relations:
\begin{equation}\label{eq:OAI1}
 \begin{cases}
 & \sum_{j=1}^n  S_j S_j^*  \le  1 \quad \text{ for all } n \in \N, \\
 & S_i^* S_i  =  \sum_{j=1}^N A(i,j) S_j S_j^* + 1 - \sum_{j=1}^N S_j S_j^* 
 \quad \text{ for all } i \text{ with } 1 \le i \le N, \\
 & S_i^* S_i  =  1 \quad \text{ for all } i \text{ with } i > N. 
 \end{cases}
\end{equation}
By \eqref{eq:OAI1} and the identification 
$$
T_i=S_i \quad \text{for } i=1,\dots,N, \qquad e_A=1-\sum_{j=1}^NS_jS_j^*,
$$
one sees that the algebra $\mathcal{O}_{A^\infty}$ contains $\TA$.
For the $N\times N$-matrix $A$ and $k\in\mathbb{N}$, 
we define an
$(N+k)\times (N+k)$ irreducible non-permutation matrix by
\[A_k:=\left[
\begin{array}{ccccccc}
&&&1&\cdots&1\\
&\text{\Huge A}&&\vdots&\vdots&\vdots\\
&&&1&\dots&1\\
1&\dots&1&1&\cdots&1\\
\vdots&\dots&\vdots&\vdots&\ddots&\vdots\\
1&\dots&1&1&\cdots&1\\
\end{array}
\right]\]
where every entries outside of $A$ is $1$.
\begin{theorem}\label{thm:mr}
\hspace{6cm}
\begin{enumerate}
\renewcommand{\theenumi}{(\roman{enumi})}
\renewcommand{\labelenumi}{\textup{\theenumi}}
\item 
The $C^*$-subalgebra of $\OAI$ 
generated by $\{S_i\}_{i=1}^{N+k}$ 
is isomorphic to the Toeplitz algebra $\mathcal{T}_{A_k}, k\geq 1$,
and the algebra $\OAI$ is the inductive limit  
$\overline{\bigcup_{k=1}^\infty\mathcal{T}_{A_k}}$ 
of the increasing sequence 
$\mathcal{T}_{A_1}\subset\mathcal{T}_{A_2}\subset\cdots$.
\item The $C^*$-algebra $\mathcal{O}_{A^\infty}$ 
is isomorphic to the Cuntz--Pimsner algebra 
$\mathcal{O}_{H_\TA}$ 
constructed from the  Hilbert $\TA-\TA$-bimodule $H_{\TA}:=F_A\otimes l^2(\N)\otimes \TA$ (see the proof of Theorem \ref{thm:construction}), 
so that $\OAI$ is a Kirchberg algebra whose $\K$-theory groups are 
\begin{equation}\label{eq:KOAITA}
(\K_0(\OAI), [1_\OAI]_0, \K_1(\OAI)) = (\K_0(\TA), [1_\TA]_0, \K_1(\TA)).
\end{equation} 
\end{enumerate}
\end{theorem}
\begin{proof}
 (i)
We write 
$P_{N+k}
:=1-\sum_{j=1}^{N+k}S_jS_j^*\in C^*(S_1,\dots , S_{N+k})\subset\OAI$,
and one has
\[
1=\sum_{j=1}^{N+k}S_jS_j^* + P_{N+k}.
\]
By \eqref{eq:OAI1},
the generators $S_1,\dots , S_N, S_{N+1}, \dots , S_{N+k}$ satisfy
\begin{equation*}
S^*_iS_i
=\sum_{j=1}^{N+k}A_k(i, j)S_jS_j^*+P_{N+k}
 \quad \text{ for } i=1,\dots,  N.
\end{equation*}
By the universality of 
$\mathcal{T}_{A_k}$,
there is a surjective homomorphism
$\mathcal{T}_{A_k}\to C^*(S_1, \dots, S_{N+k})$ 
that maps $T_j$ to $S_j$ and $e_{A_k}$ to $P_{N+k}$, 
and the map is injective because 
$0\not=P_{N+k} >S_{N+k+1}S_{N+k+1}^*$ 
and 
$ \calK \subset \mathcal{T}_{A_k}$ containing 
$e_{A_k}$ is the only non-trivial ideal of $\mathcal{T}_{A_k}$.

 (ii)
 As in \cite{Pimsner} with the proof of Theorem \ref{thm:construction}, 
Pimsner's tensor algebra $\mathcal{T}_{H_\TA}$ is generated by
$\TA$ and the left creation operators 
$T_\xi, \xi\in F_A\otimes l^2(\mathbb{N})\otimes \TA$.
By the relations 
$$
xT_\xi y=T_{(x\otimes 1\otimes 1)\xi y}, 
\quad 
T_\xi+T_\zeta=T_{\xi+\zeta}, 
\qquad 
x, y\in\TA,\; \xi, \eta\in F_A\otimes l^2(\mathbb{N})\otimes\TA
$$
seen in \cite{Pimsner},
one can check that the following elements generate 
$\mathcal{T}_{H_\TA}$:
\[\tilde{S}_i:=T_i\in \TA,\quad i=1,\dots,N, 
\qquad 
\tilde{S}_{N+k}:=T_{\Omega\otimes \delta_k\otimes 1_\TA}, \quad k\in \mathbb{N},\]
where
$
\{\delta_k\}_{k=1}^\infty$ 
is the standard basis of $\l^2(\N)$ defined by
$\delta_k(n) = 1$ if $n=k$, otherwise $0$.
We know that the above generators satisfy the three relations \eqref{eq:OAI1} by
the identities 
\[T_j^*T_{\Omega_A\otimes \delta_k\otimes 1_\TA}
=T_{(T_j^*\Omega_A)\otimes \delta_k\otimes 1_\TA}=0,\;\; \;
T_{\Omega_A\otimes \delta_k\otimes 1_\TA}^*T_{\Omega_A\otimes \delta_l\otimes 1_\TA}
=\delta_{k, l}1_\TA,\]
and  \eqref{CKT}.
Thus, the universality of $\mathcal{O}_{A^\infty}$ 
provides a surjective homomorphism:
\[\pi : \mathcal{O}_{A^\infty}\ni P_{N+k} \mapsto 1-\sum_{j=1}^{N+k}\tilde{S}_j\tilde{S}_j^*\in \mathcal{T}_{H_\TA}\cong \mathcal{O}_{H_\TA}.\]
Note that one has 
$$
1-\sum_{j=1}^{N+k}\tilde{S}_j\tilde{S}_j^*> \tilde{S}_{N+k+1}\tilde{S}^*_{N+k+1}\not=0
$$ 
and 
$\{P_{N+k}\}_{k=1}^\infty$ is not contained in the kernel $\Ker(\pi)$
of the map 
$\pi:\OAI\rightarrow \mathcal{T}_{H_\TA}$.
By (i), one has 
$\Ker(\pi)
=
\overline{\bigcup_{k=1}^\infty(\Ker(\pi)\cap\mathcal{T}_{A_k})}$ 
and every ideal  
$\Ker(\pi) \cap \mathcal{T}_{A_k}$ must be $0$ 
because 
$P_{N+k}\not\in \Ker(\pi)\cap\mathcal{T}_{A^k}$.
Now the map $\pi$ is proved to be injective.

As in the proof of Theorem \ref{thm:construction},
The embedding 
$\TA \hookrightarrow \mathcal{T}_{H_\TA}$
gives rise to a $\sqK$-equivalence such that 
$\mathcal{T}_{H_\TA}\cong \mathcal{O}_{H_\TA}$,
so that we have the $\K$-theory formula \eqref{eq:KOAITA}.
\end{proof}
In particular,
we have the following corollary.
\begin{corollary}\label{cor:whatOA}
Let $A =[A(i,j)]_{i,j=1}^N$ 
be an irreducible non-permutation matrix with entries in $\{0,1\}.$
Let $\OATI$ 
be the universal $C^*$-algebra 
generated by partial isometries 
$S_j, j \in \N$  satisfying the relations 
\eqref{eq:OAI1}
for the transposed matrix $A^t$ of $A$,
which is isomorphic to $\mathcal{O}_{H_\TAT}$.
Then the reciprocal dual $\whatOA$ of $\OA$ is realized as a corner of 
$\OATI$ such that  
\[
\whatOA \cong (1-\sum_{j=1}^NS_jS_j^*) \,  \O_{{A^t}^\infty}\,  (1-\sum_{j=1}^NS_jS_j^*).
\]
\end{corollary}
\begin{proof}
By Theorem \ref{thm:mr}, the $C^*$-algebra $\OATI$ 
is isomorphic to the Cuntz--Pimsner algebra
$\mathcal{O}_{H_{\TAT}}$constructed from the  Hilbert $C^*$-bimodule 
$H_{\TAT}$ over  $\TAT$ in the proof of Theorem \ref{thm:construction}
for the transposed matrix $A^t$ of $A$.
As in \eqref{eq:TATAT}, the pair of the Toeplitz extensions of Cuntz--Krieger algebras:
\begin{equation*}
(\TA): \, \, 0 \rightarrow \calK \rightarrow \TA \rightarrow \OA \rightarrow 0, 
\qquad
(\TAT): \, \, 0 \rightarrow \calK \rightarrow  \TAT \rightarrow \OAT \rightarrow 0
\end{equation*}
is a strong $\K$-theoretic duality pair with respect to $\epsilon = -1$ (\cite{MaJMAA2024}).
Since 
$e_{A^t}\in\TAT\subset\mathcal{O}_{H_{\TAT}}$ 
is identified with $1-\sum_{j=1}^NS_jS_j^*\in \mathcal{O}_{{A^t}^\infty}$,
Theorem \ref{thm:construction} implies that 
the reciprocal Kirchberg algebra of $\OA$ 
is given by 
$$
\widehat{\OA}
=e_{A^t}\, \mathcal{O}_{H_{\TAT}}  e_{A^t}
=(1-\sum_{j=1}^NS_jS_j^*) \, \mathcal{O}_{{A^t}^\infty}\, (1-\sum_{j=1}^NS_jS_j^*).
$$
\end{proof}

\subsection{Realization of $\OAI$ as a Exel--Laca algebra}
Recall that the  Exel--Laca algebra  
$\widetilde{O}_{\widetilde{A}}$ for a matrix 
$\widetilde{A}=[\widetilde{A}(i,j)]_{i,j\in \N}$
with $\widetilde{A}(i, j)\in \{0,1\}$ 
is  defined to be the universal unital $C^*$-algebra  
generated by a family ${s_i}, \, {i \in \N}$ of partial isometries 
subject to the relations:
\begin{enumerate}
\renewcommand{\theenumi}{(\arabic{enumi})}
\renewcommand{\labelenumi}{\textup{\theenumi}}
\item $\sum_{i=1}^n s_i s_i^*  < 1 $ for any $n \in \N$.
\item $s_i^* s_i\cdot s_j^* s_j = s_j ^*s_j \cdot s_i^* s_i$ for $i,j \in \N$.
\item  $(s_i^* s_i) s_j  =  \wtA(i,j) s_j$ for  $i,j \in \N$.
\item 
\begin{equation}\label{eq:EL(iii)}
\prod_{i \in X} s_i^* s_i
\prod_{j \in Y} ( 1- s_j^* s_j) =
\sum_{k \in \N} \wtA(X,Y,k) s_k s_k^*
\end{equation}
for all finite subsets $X, Y \subset \N$ such that
\begin{equation}\label{eq:AXYj}
\wtA(X,Y,k) = 
\prod_{i \in X} \wtA(i,k) 
\prod_{j \in Y} ( 1 - \wtA(j,k))\ne 0 
\end{equation}
for at most finitely many $k \in \N$ (\cite{EL}, \text{cf. \cite[Example 4.4.5]{RS}}).
\end{enumerate}

For a matrix $A = [A(i,j)]_{i,j=1}^N, A(i, j)\in\{0, 1\}$,
we define  $\wtAI = [\wtAI(i,j)]_{i,j\in \N}$ with co untable infinite index set $\N$ by
\begin{equation*} 
\wtAI(i,j) =
\begin{cases}
A(i,j) & \text{ if } i,j \le N, \\
1 & \text{ else. }
\end{cases}
\end{equation*} 
We will show that $\OAI$ is the Exel--Laca algebra $\wtO_{\wtAI}$
defined by the matrix $\wtAI$.
\begin{lemma}\label{lem:EL1}
For finite subsets $X, Y \subset \N$, we have
\begin{equation}\label{eq:lem1}
\{ k \in \N \mid \wtAI(X,Y,k)\ne 0 \} \subset \{1,\dots, N\}.
\end{equation}
\end{lemma}
\begin{proof}
For $k \in \N$, we have 
$\wtAI(X,Y,k)\ne 0$ if and only if 
$\wtAI(i,k) =1 $ for all $i \in X$ and 
$\wtAI(j,k) =0 $ for all $j \in Y$.
The condition
$\wtAI(j,k) =0 $ for all $j \in Y$
means that
$A(j,k) =0 $ for all $j \in Y$,
and hence we have 
that $Y \subset \{1,\dots, N\}$ and $k \le N$,
showing that \eqref{eq:lem1}.
\end{proof}
\begin{proposition}\label{prop:ExelLaca}
The $C^*$-algebra $\OAI$ is isomorphic to the Exel--Laca algebra $\wtO_{\wtAI}$
of the matrix $\wtAI$.
\end{proposition}
\begin{proof}
Since $\OAI$ is simple and universal, it is enought to show that the generators $\{s_i\}_{i\in \N}\subset \wtO_{\wtAI}$ satisfy (\ref{eq:OAI1}).
By the definition of $s_i$,
one has $s_{N+k}^*s_{N+k}s_j=s_j$ and $(s^*_{N+k}s_{N+k}s^*_j-s^*_j)(s_js^*_{N+k}s_{N+k}-s_j)=0$ for $k\in \N$,
and this implies that $\wtO_{\wtAI}$ has the unit $1_{\wtO_{\wtAI}}=s^*_js_j, j>N$.
Applying (\ref{eq:EL(iii)}) for $X=\{N+1\}, Y=\{j\}, 1\leq j\leq N$,
one has $s^*_js_j=\sum_{k=1}^N\wtAI(j, k)s_ks_k^*+1_{\wtO_{\wtAI}}-\sum_{k=1}^Ns_ks_k^*$.
\end{proof}
\begin{corollary}\label{cor:4.9}
The reciprocal dual $\whatOA$ of a simple Cuntz--Krieger algebra $\OA$ is a corner of 
the Exel--Laca algebra $\wtO_{\wttAI}$ defined by the infinite matrix $\wttAI$
for the transposed matrix $A^t$ of $A$.
\end{corollary}

\subsection{Realization of $\mathcal{O}_{A^\infty}$ as a free product} \label{subsec.freeproduct}
In this section, we will realize the $C^*$-algebra $\OAI$
as a free product $(\TA, \varphi_A)* (\OI,\varphi_{\infty})$ 
with respect to certain states on $\TA$ and $\OI$
 (see \cite{Avitzour}, \cite{Voiculescu} for free product $C^*$-algebras).

The Toeplitz algebra $\TA$ for $A =[A(i,j)]_{i,j=1}^N$  is the universal $C^*$-algebra 
generated by $N$-partial isometries $T_1,\dots, T_N$ 
and a projection $e_A$ 
subject to the relations \eqref{CKT}.
As $e_A \TA e_A = \mathbb{C} e_A$,
there exists $\varphi_A(x) \in \mathbb{C}$ for $x \in \TA$
satisfying
\begin{equation}
 e_A x e_A =\varphi_A(x)e_A, \qquad x \in \TA. 
\end{equation}
It is easy to see that 
$\varphi_A: \TA \rightarrow \mathbb{C}$
is the (vacuum) state on $\TA$ with 
$\varphi_A(e_A) =1.$ 

The Cuntz algebra $\OI$ is the universal 
$C^*$-algebra generated by isometries 
$s_i, i \in \N$ satisfying the relations:
$ \sum_{j=1}^n s_j s_j^* < 1$
for
$ n \in \N$
(\cite{Cuntz77}). 
For a word $\xi = (\xi_1,\dots, \xi_k)$ with $\xi_i \in \N$,
we write $|\xi | = k$ and $s_{\xi_1}\cdots s_{\xi_k}$ as $s_\xi$.
Let us denote by $\OalgI$ the $*$-subalgebra of $\OI$ algebraically generated by 
$s_i, i \in \N$ 
and the unit $1$ of $\OI$.
Hence 
for 
$x \in \OalgI$ there exists a finite set 
$F \subset \cup_{n=0}^\infty \{1,2,\dots \}^n$
of finite words of $\N$ and $c_0, c_{\xi,\eta} \in \mathbb{C}$ 
 such that 
\begin{equation}\label{eq:OIdecomposition}
x = \sum_{\xi, \eta \in F,\; |\xi | \ge 1 \text{, } |\eta | \ge 1} c_{\xi, \eta}
s_\xi s_\eta^* + c_0 1
\end{equation}
The vacuum state of the Fock space for $\OI$ gives a state $\varphi_\infty: \OI\rightarrow \mathbb{C}$ satisfying
$ x \in \OalgI \rightarrow c_0\in \mathbb{C}.$

\medskip

We will show that the tensor algebra 
$\mathcal{T}_{H_\TA}$ defined in the proof of Theorem \ref{thm:mr} (ii),
 which is isomorphic to $\mathcal{O}_{A^\infty}$, 
is identified with the reduced free product 
$(\TA, \varphi_A)*(\mathcal{O}_\infty, \varphi_\infty).$ 
Let us denote by 
$P_{\TA}: F(H_\TA)\to (H_\TA)^{\otimes 0}=\TA$ 
the projection
on the Fock space $F(H_\TA):=\bigoplus_{n=0}^\infty H_{\TA}^{\otimes_{\TA}n}$.
We define a state 
$\varphi_{A^\infty} : \mathcal{T}_{H_\TA}\to \mathbb{C}$ 
by the composition
\[\varphi_{A^\infty} : \mathcal{T}_{H_\TA}
\xrightarrow
{E_A}\TA\xrightarrow{\varphi_A}\mathbb{C},\]
where 
$E_A: \mathcal{T}_{H_\TA} \rightarrow \TA$
is the conditional expectation
defined by
$E_A(x) = P_{\TA} x P_{\TA}, \, x \in \T_{H_{\TA}}$
(see \cite{Kumjian}, \cite{Pimsner}).
We identify the subalgebra of 
$\mathcal{T}_{H_\TA}$ generated by the image of the left action of
$\TA$ on $H_{\TA}$ 
(resp. 
$\{s_k \}_{k=1}^\infty:=\{T_{\Omega\otimes \delta_k\otimes 1_\TA}\}_{k=1}^\infty$
) 
with $\TA$ (resp. $\mathcal{O}_\infty$).
Since 
$\varphi_{A^\infty}(x)e_A = e_A E_A(x)e_A$ 
and 
$P_{\TA} T_{\Omega\otimes \delta_k\otimes 1_{\TA}}=P_{\TA}s_k=0$,
one has 
$$\varphi_{A^\infty} |_\TA=\varphi_A,\quad \varphi_{A^\infty} |_{\mathcal{O}_\infty}=\varphi_\infty.$$
Then, the standard argument of the reduced free product 
(see \cite[Chapter 4.7]{BO}) shows the following.
\begin{theorem}\label{thm:freeproduct}
The reduced free product $(\TA, \varphi_A)*(\mathcal{O}_\infty, \varphi_\infty)$ 
is isomorphic to $(\mathcal{T}_{H_\TA}, \varphi_{A^\infty})$.
\end{theorem}
\begin{proof}
One has
\begin{align*}
\Ker(\varphi_A) 
=& \{x-\varphi_A(x)\;|\;x\in\TA\}
= \overline{\Span \{T_\mu T^*_\nu\;|\; |\mu|>0 \text{ or } |\nu|>0 \}} \\
\Ker(\varphi_\infty) 
=& \{x-\varphi_\infty(x)\;|\;x\in\mathcal{O}_\infty\}
= \overline{\Span \{s_\xi s^*_\eta\;|\; |\xi|>0 \text{ or } |\eta|>0 \}}
\end{align*}
and it is enough to show the freeness of 
$\Ker(\varphi_A)$ 
and 
$\Ker(\varphi_\infty)$ 
with respect to the state $\varphi_{A^\infty}$ by \cite[Theorem 4.7.2]{BO}.
For
$$
x_i\in \{T_\mu, T_\mu T_\nu^*, T_\nu^* \;|\; |\mu|, |\nu|>0\}
\quad \text{ and }
\quad
y_j\in \{s_\xi, s_\xi s_\eta^*, s_\eta^*\; |\; |\xi|, |\eta|>0\},
$$
we will prove 
\[\varphi_{A^\infty}(x_1y_1x_2y_2\cdots z_n)=\varphi_{A^\infty}(y_1x_1y_2x_2\cdots z_n)=0,\]
where $z_n = x_n$ or $y_n$ for some $n \in \N$. 
Recall that 
$e_A E_A(T_iT_i^*)e_A
=e_A T_iT_i^* e_A=0
$ 
and 
$E_A(s_js_j^*)=P_\TA s_js_j^*P_\TA=0$.
Since 
$
e_A P_\TA T_i=P_\TA s_j=0$ 
and the projections 
$\{ T_\mu T_\mu^*, \, s_\xi s_\xi^* \mid |\mu|, |\xi|>0\}$ 
are mutually orthogonal,
the value 
$$
\varphi_{A^\infty}(x_1y_1x_2y_2\cdots z_n)e_A
=
e_A E_A(x_1y_1\cdots z_n)e_A
$$ 
is 0 unless 
$x_1=T_{\mu}^*, y_1=s_{\xi}^*, x_2=T_{\nu}^*, \cdots$,
and it is easy to see 
$
\varphi_{A^\infty}(T_\mu^*s_{\xi}^*\cdots T_{\nu'}^*)e_A
=\varphi_{A^\infty}(T_\mu^*s_{\xi}^*\cdots s_{\eta'}^*)e_A
=0.
$
Thus, one has 
$\varphi_{A^\infty}(x_1y_1\cdots z_n)=0$ 
and the same argument shows $\varphi_{A^\infty}(y_1x_1\cdots z_n)=0$.
Hence by \cite[Theorem 4.7.2.(3)]{BO},
 one has 
\[(\mathcal{T}_{H_\TA}, \varphi_{A^\infty})\cong (\TA, \varphi_A)*(\mathcal{O}_\infty, \varphi_\infty).\]
\end{proof}
Under the identification $\mathcal{T}_{H_\TA}=\OAI$,
we regard $\varphi_{A^\infty}$ as a state of $\OAI$
which we will see again in Subsection \ref{subsec:Ground}.

\section{Generators and relations of  $\whatOA$} \label{sect:generatorsrelations}

\subsection{Generators of  $\whatOA$}

In this section, we find generators of $\whatOA$ and their relations.
Let $S_i, i \in \N$  be the canonical generating partial isometries of $\OATI$ satisfying
\begin{equation}\label{eq:RelationOAI}
 \begin{cases}
 & \sum_{j=1}^n  S_j S_j^*  \le  1 \quad \text{ for all } n \in \N, \\
 & S_i^* S_i  =  \sum_{j=1}^N A(j,i) S_j S_j^* + 1 - \sum_{j=1}^N S_j S_j^* 
 \quad \text{ for all } i \text{ with } 1 \le i \le N, \\
 & S_i^* S_i  =  1 \quad \text{ for all } i \text{ with } i > N, 
 \end{cases}
\end{equation}
which are the relations \eqref{eq:OAI1} 
for the transposed matrix $A^t$ for the given matrix $A=[A(i,j)]_{i,j=1}^N.$
We put
\begin{equation}\label{eq:PN}
P_N =  1 - \sum_{j=1}^N S_j S_j^*.
\end{equation}
We identify $\whatOA$ with the corner $P_N \OATI P_N$ of $\OATI$.
Define partial isometries $T_{n,0}, T_{n,i}$ for $n>N, \, i \in\N$ in $\OATI$
by
\begin{equation}\label{eq:RelationOAITS}
 \begin{cases}
  T_{n,0} & =  S_n P_N \quad \text{ for } n >N, \\
  T_{n,i}  & = S_n S_i S_n^* \quad \text{ for } n >N \text{ and } i \in \N.  
 \end{cases}
\end{equation}
\begin{lemma}\label{lem:generator2}
The partial isometries
$
T_{n,i}, \,  T_{n,0} 
$ for 
$ n >N,  \, i \in \N$
belong to $P_N \OATI P_N$,
and 
generate the $C^*$-algebra  $P_N \OATI P_N$.
\end{lemma}
\begin{proof}
Since
we have
\begin{equation*}
P_N S_n  = 
\begin{cases}
S_n & \text{ if } n>N,\\
0 &   \text{ if } 1\le n \le N,
\end{cases}
\qquad
S_n^* P_N  = 
\begin{cases}
S_n^* & \text{ if } n>N,\\
0 & \text{ if } 1 \le n \le N,
\end{cases}
\end{equation*}
the identities   
$T_{n,i} = P_N T_{n,i} P_N$ and $T_{n,0} = P_N T_{n,0} P_N$
hold so that
$
T_{n,i}, \,  T_{n,0}
$
belong to 
$ P_N \OATI P_N.
$

We will next show that 
$
T_{n,i}, \,  T_{n,0}, 
 n >N,  \, i \in \N$
generate
$ P_N \OATI P_N.
$
Let $\N^*$ be the set $\cup_{n=1}^\infty \N^n$ of finite words of $\N$.
Since the $*$-subalgebra of $\OATI$ generated algebraically by
$S_i, S_i^*, i \in \N$ is dense in $\OATI$ and 
any polynomial of $S_i, S_i^*, i \in \N$
is a finite linear combination of the elements  
$$
1,\quad 
S_\mu S_\nu^* = S_{\mu_1}\cdots S_{\mu_m} S_{\nu_n}^* \cdots S_{\nu_1}^*, \quad S_\mu, \quad S_\nu^*,
\quad \mu=(\mu_1,\dots, \mu_m), \, 
\nu = (\nu_1,\dots,\nu_n) \in \N^*,
$$
it is enough to show that $T_{n, i}$ generate 
 $P_N,\; P_N S_\mu S_\nu^* P_N,\; P_NS_\mu P_N,\; P_NS^*_\nu P_N$.
One has $T_{n, 0}^*T_{n, 0}=P_NS_n^*S_nP_N=P_N$.

The element
$$
 P_N S_\mu S_\nu^* P_N
 =P_N S_{\mu_1}\cdots S_{\mu_m} S_{\nu_n}^* \cdots S_{\nu_1}^* P_N
 $$
 is non-zero if and only if $\mu_1 >N$ and $\nu_1 >N$.
 
 Let
$ \mu=(\mu_1,\dots, \mu_m), \, 
\nu = (\nu_1,\dots,\nu_n) \in \N^*.
$
 We then have three cases.

Case 1. 
 $m, n \ge 2$:
 Since $
 P_N S_\mu S_\nu^* P_N =S_\mu S_\nu^*
 $
  with $\mu_1 , \nu_1>N$, we have
 $S_{\mu_1}^* S_{\mu_1} = S_{\nu_1}^* S_{\nu_1} =1$,
 and
$
S_{\mu_1} S_{\nu_1}^* 
=  S_{\mu_1} S_{\nu_1}^* P_N 
=  S_{\mu_1} S_{\nu_1}^* S_{\mu_1}^* S_{\mu_1} P_N
= T_{\mu_1,\nu_1}^* T_{\mu_1, 0}
$
 so that 
\begin{align*}
   S_\mu S_\nu^* 
= & S_{\mu_1}S_{\mu_2}\cdots S_{\mu_m} S_{\nu_n}^* \cdots S_{\nu_2}^*S_{\nu_1}^* \\
= & S_{\mu_1}S_{\mu_2}\cdot S_{\mu_1}^* S_{\mu_1} \cdot S_{\mu_3} 
\cdots S_{\mu_m} \cdot S_{\mu_1}^* S_{\mu_1} 
\cdot S_{\nu_n}^* \cdot S_{\mu_1}^* S_{\mu_1} \cdot S_{\nu_{n-1}}^* 
\cdots S_{\nu_2}^*\cdot S_{\mu_1}^* S_{\mu_1} \cdot S_{\nu_1}^* \\
= & T_{\mu_1,\mu_2} T_{\mu_1,\mu_3} 
\cdots T_{\mu_1,\mu_m}  \cdot  
T_{\mu_1, \nu_n}^* T_{\mu_1,\nu_{n-1}}^* 
\cdots T_{\mu_1,\nu_2}^* \cdot  S_{\mu_1} S_{\nu_1}^* \\
= & T_{\mu_1,\mu_2} T_{\mu_1,\mu_3} 
\cdots T_{\mu_1,\mu_m}  \cdot  
T_{\mu_1, \nu_n}^* T_{\mu_1,\nu_{n-1}}^* 
\cdots T_{\mu_1,\nu_2}^* \cdot T_{\mu_1,\nu_1}^* T_{\mu_1, 0}.
\end{align*}

Case 2. 
 $m=1,\,  n \ge 2$ or $m\ge 2, \, n=1$:
 We may assume that $m=1,\,  n \ge 2.$ 
As $S_\mu = S_{\mu_1}$, 
  we have
\begin{align*}
   S_\mu S_\nu^* 
= & S_{\mu_1} S_{\nu_n}^* \cdots S_{\nu_2}^*S_{\nu_1}^* \\
= & S_{\mu_1} S_{\nu_n}^* \cdot S_{\mu_1}^* S_{\mu_1} \cdot S_{\nu_{n-1}}^* 
\cdots S_{\nu_2}^*\cdot S_{\mu_1}^* S_{\mu_1} \cdot S_{\nu_1}^* \\
= & T_{\mu_1, \nu_n}^* T_{\mu_1,\nu_{n-1}}^* 
\cdots T_{\mu_1,\nu_2}^* \cdot  S_{\mu_1} S_{\nu_1}^* \\ 
= &  T_{\mu_1, \nu_n}^* T_{\mu_1,\nu_{n-1}}^* 
\cdots T_{\mu_1,\nu_2}^* \cdot  T_{\mu_1,\nu_1}^* T_{\mu_1, 0}. 
\end{align*}

Case 3. 
 $m=n=1$:
As $S_\mu = S_{\mu_1}, S_\nu = S_{\nu_1}$, 
  we have
\begin{equation*}
   S_\mu S_\nu^* 
  =  S_{\mu_1} S_{\nu_1}^* P_N 
  =   S_{\mu_1} S_{\nu_1}^* S_{\mu_1}^*S_{\mu_1} P_N
  = T_{\mu_1,\nu_1}^* T_{\mu_1, 0}.
\end{equation*} 
  
The same argument shows that $P_NS_\mu P_N, P_NS_\nu^*P_N$ are generated by $T_{n, i}$.  
We thus conclude that 
$
T_{n,i},  \, T_{n,0}
$
for 
$ n >N, \, i \in \N
$ 
generate  $P_N \OATI P_N$.
\end{proof}
Put
\begin{equation*}
R_i:=  S_{N+1}S_i S_{N+1}^*, \quad
S: = S_{N+1} P_N, \qquad i \in \N.
\end{equation*}
\begin{lemma}\label{lem:generator3}
The $C^*$-subalgebra $C^*(R_i, S \mid i \in \N)
$ of $\OATI$ generated by
$R_i, S, i \in \N$
coincides with
$P_N \OATI P_N$.
\end{lemma}
\begin{proof}
Put
$U_n := S_n S_{N+1}^*
$
for 
$
n >N.
$
We have the identities
\begin{align*}
T_{n,i} & = U_n R_i U_n^* \quad (\text{and hence }\,R_i= U_n^* T_{n,i} U_n) \text{ for } i \in \N, \, n >N, \\
T_{n,0} & = U_n S \quad (\text{and hence } S = U_n^* T_{n,0}) \text{ for }n >N.
\end{align*}
so that 
 the $C^*$-subalgebra 
$C^*(R_i, S, U_n \mid i \in \N, n>N)$
of $\OATI$
generated by 
$R_i, S, U_n, i \in \N, n>N$
coincides with 
$P_N \OATI P_N$.
As the identity 
$
S^* R_n = U_n
$
for $n>N$
holds,
the partial isometries 
$U_n, n>N$ are written by $R_i, S, i \in \N,$
so that  
$C^*(R_i, S \mid i \in \N)$
coincides with 
$P_N \OATI P_N$.
\end{proof}
Therefore the $C^*$-algebra 
$ P_N \OATI P_N$ is generated by the sequence of partial isometries 
$R_i,  i \in \N$
with one additional partial isometry $S$.
\begin{proposition}\label{prop:generator4}
Keep the above notation.
We have the following relations:
\begin{enumerate}
\renewcommand{\theenumi}{(\arabic{enumi})}
\renewcommand{\labelenumi}{\textup{\theenumi}}
\item
$\sum_{j=1}^m R_j R_j^* < R_{N+1}^* R_{N+1} =R_{N+2}^* R_{N+2} =\cdots $ for $m \in \N$.
\item
$R_i^* R_i = \sum_{j=1}^N A(j,i) R_j R_j^* + R_{N+1}^* R_{N+1} -\sum_{j=1}^N R_j R_j^*$ for $1 \le i \le N$.
\item
$SS^* = R_{N+1}^* R_{N+1} -\sum_{j=1}^N R_j R_j^*, \quad
 S^* S = P_N.$ 
\item $S^* R_{N+1} = R_{N+1}^* R_{N+1}, 
\quad S R_{N+1}^* = R_{N+1} R_{N+1}^*$.
\end{enumerate}
\end{proposition}
\begin{proof}
We know  all the relations (1),(2), (3), and (4) from \eqref{eq:RelationOAI}.
\end{proof}
\subsection{Universal concrete relations of $\whatOA$}
For an $N\times N$ irreducible matrix $A = [A(i,j)]_{i,j=1}^N$
with entries in $\{0,1\},$
let $r_i,i \in \N,\,  s, \, p_1$ be  a countable family of partial isometries 
in a unital $C^*$-algebra 
satisfying the relations: 

\medskip

$(R1) \quad r_i^* r_i = \sum_{j=1}^N r_j r_j^* + p_1 \quad \text{ for } i >N. $

$(R2) \quad r_i^* r_i = \sum_{j=1}^N A(j,i) r_j r_j^* + p_1 \quad \text{ for } 1 \le i \le N.$ 

$(R3) \quad \sum_{j=N+1}^m r_j r_j^*  < p_1 \quad \text{ for } m > N.$

$(R4) \quad s s^* = p_1, \quad s^* s = 1.$  

$(R5) \quad s^* r_{N+1} = r_{N+1}^* r_{N+1}, \quad s r_{N+1}^* = r_{N+1} r_{N+1}^*.$

\medskip
We note that the second relation 
$s r_{N+1}^* = r_{N+1} r_{N+1}^*$ of $(R5)$
follows from the first relation $s^* r_{N+1} = r_{N+1}^* r_{N+1}$
of $(R5)$ and $(R3), (R4)$.
Because by the first relation 
we have
$ss^* r_{N+1} r_{N+1}^* = s r_{N+1}^*,$
and
$(R3)$ and $(R4)$ imply that  
 $ss^* r_{N+1} r_{N+1}^* = r_{N+1} r_{N+1}^*.$ 

\medskip 
Let us denote by 
$C^*(r_i, s \mid i \in \N)$ the unital $C^*$-algebra
generated by the partial isometries $r_i, s, \, i \in \N$.
By $(R1)$, the projection $p_1$ is written in terms of $r_j, j =1,\dots,N$,
so that it belongs to $C^*(r_i, s \mid i \in \N)$. 
Keep the situation.
\begin{lemma} \label{lem:C*(riss*ri)}
Let $C^*(r_1,\dots, r_N, s, s^* r_i  \mid i \ge N+2)$
be the $C^*$-subalgebra of  $C^*(r_i, s \mid i \in \N)$ 
generated by 
$r_1,\dots, r_N, s, s^* r_i$ for $i \ge N+2$.
Then we have
\begin{equation}\label{eq:7.13}
C^*(r_1,\dots, r_N, s, s^* r_i  \mid i \ge N+2)
=C^*(r_i, s \mid i \in \N).
\end{equation}
\end{lemma}
\begin{proof}
By $(R3)$ and $(R4)$,  
for $i \ge 1$ we have
$ r_{N+i} =ss^* r_{N+i} =s\cdot s^* r_{N+i}.$
Hence for $i \ge 2$ we have 
$ r_{N+i} \in C^*(r_1,\dots, r_N, s, s^* r_i  \mid i \ge N+2)$.
By $(R5)$ and $(R1)$,  we have
\begin{align*}
r_{N+1} = s s^* r_{N+1} =s\cdot r_{N+1}^* r_{N+1} = s r_{N+2}^* r_{N+2}
\end{align*}
so that  $ r_{N+1} \in C^*(r_1,\dots, r_N, s, s^* r_i  \mid i \ge N+2)$,
showing the equality \eqref{eq:7.13}.
\end{proof}
\begin{lemma}
The identity
\begin{equation} \label{eq:realtion7.31}
r_{N+i}^* s s^* r_{N+i} = \sum_{j=1}^N r_j r_j^* + s s^*\quad \text{ for } i \in \N 
\end{equation}
holds.
\end{lemma}
\begin{proof}
For $i \in \N$, we have 
$ss^* r_{N+i} = r_{N+i}$
and hence 
$r_{N+i}^* ss^* r_{N+i} = r_{N+i}^*r_{N+i}.$
We then have 
\eqref{eq:realtion7.31}
by $(R1)$ and $(R4).$
\end{proof}
\begin{lemma}\label{lem:7.15}
 Keep the above situation.
Put
\begin{equation*}
t_i : =
\begin{cases}
r_i & \text{ for } 1\le i \le N, \\
s & \text{ for } i = N+1, \\
s^* r_i & \text{ for } i \ge N+2.
\end{cases}
\end{equation*}
We then have
\begin{enumerate}
\renewcommand{\theenumi}{(\roman{enumi})}
\renewcommand{\labelenumi}{\textup{\theenumi}}
\item
\begin{equation} \label{eq:ti*ti}
t_i^* t_i =
\begin{cases}
\sum_{j=1}^N A(j,i) t_j t_j^* + t_{N+1} t_{N+1}^* &  \text{ for } 1\le i \le N, \\
1 & \text{ for } i = N+1, \\
\sum_{j=1}^N  t_j t_j^* + t_{N+1} t_{N+1}^* & \text{ for } i \ge N+2.
\end{cases}
\end{equation}
\item
\begin{equation}\label{eq:1>titi*}
1 > \sum_{j=1}^N t_j t_j^* + t_{N+1} t_{N+1}^* + \sum_{j=N+2}^m t_j t_j^*,  \qquad m \ge N+2.
\end{equation}
\end{enumerate}
\end{lemma}
\begin{proof}
(i) The formula \eqref{eq:ti*ti} 
is direct from $(R2), (R4)$ and \eqref{eq:realtion7.31}.

(ii) For $m \ge N+2$, we have \eqref{eq:1>titi*} by using (R5) in the following way:
\begin{align*}
  & \sum_{j=1}^N t_j t_j^* + t_{N+1} t_{N+1}^* + \sum_{j=N+2}^m t_j t_j^* \\ 
= &  \sum_{j=1}^N r_j r_j^* + p_1 + \sum_{j=N+2}^m s^* r_j r_j^* s 
=   r_{N+1}^* r_{N+1} + s^* (\sum_{j=N+2}^m  r_j r_j^*) s \\
= &  s^* r_{N+1} r_{N+1}^* s + s^* (\sum_{j=N+2}^m  r_j r_j^*) s 
=   s^*  (\sum_{j=N+1}^m  r_j r_j^*) s \\
< &  s^* p_1 s \le 1.
\end{align*}
\end{proof}
Define the matrix
$
\widehat{A}_\infty 
= [\widehat{A}_\infty(i,j)]_{i,j\in \N}
$ 
over $\N$ with entries in $\{0,1\}$ by setting
\begin{equation}\label{eq:matrixAinfty}
\widehat{A}_\infty (i,j) :=
\begin{cases}
A(j,i) & \text{ if } 1 \le i , j \le N, \\
1 & \text{ if } 1 \le i \le N,\, \, j= N+1, \\ 
0 & \text{ if } 1 \le i \le N,\, \, j\ge  N+2, \\ 
1 & \text{ if }  i =N+1, \\ 
1 & \text{ if }  i \ge N+2,\, \, j \le  N+1, \\ 
0 & \text{ if }  i \ge N+2,\, \, j \ge N+2.
\end{cases}
\end{equation}
The matrix $\widehat{A}_\infty$ is written as  
\begin{equation*}
\widehat{A}_\infty=
\left[
\begin{array}{ccccccc}
&                &1      &0     &0     &\cdots \\
&\text{\Huge A}^t&\vdots&\vdots&\vdots&\vdots \\
&                &1      &0     &0     &\cdots \\
1&\dots          &1      &1     &1     &\cdots \\
1&\dots          &1      &0     &0     &\cdots \\
1&\dots          &1      &0     &0     &\cdots \\
\vdots &\dots    &\vdots &\vdots&\vdots&\vdots \\
\end{array}
\right]
\end{equation*}
The following lemma is straightforward from Lemma \ref{lem:7.15}
\begin{lemma}\label{lem:7.16}
Let $\widehat{A}_\infty$ be the matrix defined in \eqref{eq:matrixAinfty}.
Then we have for $i,j \in \N$
\begin{equation}
\widehat{A}_\infty(i,j) = 
\begin{cases}
1 & \text{ if }\, \,  t_i^* t_i \ge t_j t_j^*, \\
0 & \text{ otherwise}.
\end{cases}
\end{equation}
\end{lemma}
\begin{proposition}\label{prop:universalExelLaca}
Keep the above notation.\hspace{2cm}
\begin{enumerate}
\renewcommand{\theenumi}{(\roman{enumi})}
\renewcommand{\labelenumi}{\textup{\theenumi}}
\item The $C^*$-algebra $C^*(r_i, s\mid i \in \N)$ coincides with the $C^*$-algebra 
$C^*(t_i\mid i \in \N)$ generated by the partial isometries $t_i, i \in \N$. 
\item The $C^*$-algebra  
$C^*(t_i\mid i \in \N)$ 
is the Exel--Laca algebra $\mathcal{O}_{\widehat{A}_\infty}$
defined by the infinite matrix $\widehat{A}_\infty$. 
\item
The Exel--Laca algebra $\mathcal{O}_{\widehat{A}_\infty}$ 
is a unital simple purely infinite Kirchberg algebra.
\end{enumerate}
\end{proposition}
\begin{proof}
(i) It is direct from Lemma \ref{lem:C*(riss*ri)} and the definition of the partial isometries $t_i, i \in \N.$

(ii) The assertion easily follows from Lemma \ref{lem:7.15} and Lemma \ref{lem:7.16}. 

(iii) The associated directed graph with the matrix 
$\widehat{A}_\infty$ is transitive and has one loop having an exit.
By \cite[P. 82]{RodSto}, the  Exel--Laca algebra $\mathcal{O}_{\widehat{A}_\infty}$
is  a unital simple purely infinite Kirchberg algebra.
\end{proof}
Therefore we conclude the following theorem.
\begin{theorem}\label{thm:univrelation}
Let $A =[A(i,j)]_{i,j=1}^N$ be an irreducible non-permutation matrix with entries in 
$\{0,1\}.$
The the reciprocal dual $\whatOA$ of the Cuntz--Krieger algebra $\OA$
is the universal concrete $C^*$-algebra $C^*(r_i, s\mid i\in \N)$
generated by a family of partial isometries $r_i, \, i \in \N$ with one isometry $s$
subject to the following relations: 

\medskip

$(1) \quad \sum_{j=1}^m r_j r_j^*  < r_{N+1}^* r_{N+1} =  r_{N+2}^* r_{N+2} = \dots $ for $m \in \N$.

$(2) \quad r_i^* r_i 
= \sum_{j=1}^N A(j,i) r_j r_j^* + r_{N+1}^* r_{N+1} - \sum_{j=1}^N r_j r_j^*$ 
for $ 1 \le i \le N.$ 

$(3) \quad s s^* = r_{N+1}^* r_{N+1} - \sum_{j=1}^N r_j r_j^*, \quad s^* s = 1.$  

$(4) \quad s^* r_{N+1} = r_{N+1}^* r_{N+1}, \quad s r_{N+1}^* = r_{N+1} r_{N+1}^*.$
\end{theorem}
\begin{proof}
Let $r_i, s, (i \in \N),$ be partial isometries satisfying the relations
(1), (2), (3), (4).
The relations (1), (2), (3), (4) yield  relations 
$(R1), (R2), (R3), (R4), (R5)$
by putting
$p_1: = r_{N+1}^* r_{N+1} - \sum_{j=1}^N r_j r_j^*$.
Hence  two kinds of the relations (1), (2), (3), (4)
and $(R1), (R2), (R3), (R4), (R5)$ are equivalent.
Let 
$C_{\text{univ}}^*(r_i, s\mid i\in \N)$ be the universal unital $C^*$-algebra 
generated by a family of partial isometries $r_i, \, i \in \N$ with one isometry $s$
subject to the above  relations $(1), (2), (3), (4)$.
Let us realize the reciprocal dual $\whatOA$ of $\OA$ as the corner 
$P_N \OATI P_N$ of $\OATI$.
As in the preceding discussion, 
the algebra 
$P_N \OATI P_N$ is generated by the family 
$R_i: = S_{N+1} S_i S_{N+1}^*, S: = S_{N+1}P_N, \, i \in \N$
of partial isometries, where $S_i, i \in \N$ are the canonical generating partial 
isometries of $\OATI$ satisfying \eqref{eq:RelationOAI}.
By putting $P_1 = R_i^* R_i - \sum_{j=1}^N A^t(i,j) R_j R_j^*$ for $1 \le i \le N$,
one sees that 
the operators $R_i, S, \, i \in \N$ 
satisfy the relations $(R1), (R2), (R3), (R4), (R5).$
By the universality of the $C^*$-algebra  
$C_{\text{univ}}^*(r_i, s\mid i\in \N)$
the correspondence
$r_i \rightarrow R_i, \, \, s\rightarrow S$ yields a surjective $*$-homomorphism
from $C_{\text{univ}}^*(r_i, s\mid i\in \N)$ to $P_N \OATI P_N$ written
$\Phi$.
By Proposition \ref{prop:universalExelLaca}, we know that  
$C_{\text{univ}}^*(r_i, s\mid i\in \N)$ is the Exel--Laca algebra 
$\mathcal{O}_{\widehat{A}_\infty}$ of the infinite matrix $\widehat{A}_\infty$.
Since $\mathcal{O}_{\widehat{A}_\infty}$ is simple, 
the surjective homomorphism
$\Phi:C_{\text{univ}}^*(r_i, s\mid i\in \N)\rightarrow P_N \OATI P_N$
is actually isomorphic, so that we conclude that 
 $\mathcal{O}_{\widehat{A}_\infty}\cong C_{\text{univ}}^*(r_i, s\mid i\in \N)$ is isomorphic to $\whatOA$, and hence 
 the relations $(1), (2), (3), (4)$  are universal relations 
 defining the $C^*$-algebra $\whatOA$.
\end{proof}

\section{Gauge actions on  $\OAI$ and $\whatOA$ and their ground states}\label{sect:gaugeaction}
The gauge actions on Cuntz--Krieger algebras are most important actions of the circle group 
$\mathbb{T} (= \mathbb{R}/\Z)$
 to analyze the structure of the algebras as in \cite{CK}.
In this section, we will define and study such actions  on the reciprocal duals
$\whatOA$.

\subsection{Three kinds of Gauge actions on $\OAI$ and $\whatOA$}
Let us regard $\OAI$ as  the universal $C^*$-algebra generated by
a family of partial isometries $S_i, i \in \N$
 satisfying the relations \eqref{eq:OAI1}.
 Let $P_N$ be the projection $ 1 -\sum_{j=1}^N S_j S_j^*$
   in $\OAI$ defined by \eqref{eq:PN}.
The first two relations of \eqref{eq:OAI1}
yield the relations
\begin{equation}\label{eq:TA666}
     \sum_{j=1}^N   S_j S_j^*  + P_N =  1, \qquad
 S_i^* S_i  =  \sum_{j=1}^N A(i,j) S_j S_j^* + P_N. 
\end{equation}
Recall that the reciprocal dual $\whatOA$ of the Cuntz--Krieger algebra $\OA$
is regarded as  the corner $\whatOA = P_N \OATI P_N$ of the algebra $\OATI$
for the  the transposed 
matrix $A^t$ of $A$ instead of $A$. 
It seems to be natural to consider the following three gauge actions on $\OAI$ and $\whatOA$.
We first define three kinds of gauge actions $\alpha, \beta, \gamma $ on $\OAI$, which are actions of 
$\bbT = \R/\Z$, by setting
\begin{align*}
\alpha_t(S_j)  
= & 
{\begin{cases}
e^{2\pi i t} S_j & \text{ if } 1\le j \le N, \\
S_j  &\text{if } N < j, 
\end{cases}
} \\
\beta_t(S_j)  
= & 
{\begin{cases}
 S_j & \text{ if } 1\le j \le N, \\
e^{2\pi i t} S_j  &\text{if } N < j, 
\end{cases}
} \\
\gamma_t(S_j)  
=\, \,  & e^{2\pi i t} S_j \quad \text{ for all } j \in \N,
\end{align*}
so that 
$\gamma_t = \alpha_t \circ \beta_t =\beta_t \circ \alpha_t$.
Since 
$\alpha_t(P_N) = \beta_t(P_N) = \gamma_t(P_N) = P_N, t \in \R$,
by considering the transposed matrix $A^t$ instead of $A$,
they all yield actions of $\whatOA$,
which are written as $\hat{\alpha}^A, \hat{\beta}^A, \hat{\gamma}^A$,
respectively.
They also give rise to elements of the fundamental groups
$\pi_1(\Aut(\OAI))$ and $\pi_1(\Aut(\whatOA))$ in natural ways,
 respectively.
\begin{proposition}\label{prop:homotopybeta}
The homotopy class $[\beta]$
of the $\beta$-action of $\mathbb{T}$
is trivial in $\pi_1(\Aut(\OAI))$ 
and $[\hat{\beta}^A]$ is trivial in $\pi_1(\Aut(\whatOA))$.
\end{proposition}
\begin{proof}
We will first prove that $[\beta] =0$ in $\pi_1(\Aut(\OAI))$.
One may assume that $\OAI$ is represented on the Hilbert $C^*$-bimodule 
$F_A\otimes \ell^2(\N) \otimes \TA$ over $\TA$.
Let $U_t(n), n \in \N$ be the unitary on 
$F_A\otimes \ell^2(\N) \otimes \TA$
defined by 
\begin{equation*}
U_t(n)(\xi \otimes e_k\otimes x) 
=
\begin{cases} 
  e^{2\pi i t}\xi \otimes e_k\otimes x & \text{ if } n=k, \\
\xi \otimes e_k\otimes x & \text{ if } n\ne k
\end{cases}
\end{equation*}
for $\xi \otimes e_k\otimes x \in F_A\otimes \ell^2(\N) \otimes \TA$.
Let $\OalgAI$ be the dense $*$-algebra algebraically generated by 
$S_i, i \in \N$.
For a word $\mu = (\mu_1,\dots,\mu_m) \in \N^m$,
 let us denote by
$| \mu |^N = | \{ \mu_j \mid N < \mu_j \}|.$
Any element of $\OalgAI$ is a finite linear combination of elements 
of the form 
$S_{\mu_1} \cdots S_{\mu_m} S_{\nu_n}^* \cdots S_{\nu_1}^*$.
By rotating, we know that 
$\begin{bmatrix}
e^{2\pi i t} & 0 \\  
0 & 1 
\end{bmatrix}
$ is homotopic to 
$\begin{bmatrix}
1 & 0 \\
0 & e^{2\pi i t}
\end{bmatrix}
$
and hence 
$\Ad (U_t(n))(S_n) $ is homotopic to 
$\Ad (U_t(n+1))(S_n)= S_n $.
This shows that 
$\beta_t(X)$ is homotopic to $X$
for any $X \in \OalgAI$,
 showing that 
the class $[\beta]$ in $\pi_1(\Aut(\OAI))$ 
is  $[\id] $ which is  trivial.
Similarly we know that 
the class $[\beta]$ in $\pi_1(\Aut(\whatOA))$ 
is $[\id] $ which is  trivial.
\end{proof}
\begin{corollary}\label{cor:homotopyalpha=gamma}
$[\alpha] = [\gamma] $ in $\pi_1(\Aut(\OAI))$ 
and  $[\hat{\alpha}^A] = [\hat{\gamma}^A] $ 
in $\pi_1(\Aut(\whatOA))$.
\end{corollary}

\subsection{Ground states}\label{subsec:Ground}
Let $\A$ be a $C^*$-algebra with a one-parameter group action 
$\alpha: \R \rightarrow \Aut(\A)$.
Let us denote by $\A_a$ the set of analytic elements with respect to $\alpha$
which is defined by
\begin{align*}
\A_a := \{ X \in \A \mid t \in \R \rightarrow \alpha_t(X) \in \A 
\text{ can be extended to an entire analytic function on } \mathbb{C}  \}.
\end{align*}
Let us denote by $\mathbb{C}_+$ 
the upper half plane
$\{ z = t + is \in \mathbb{C} \mid s >0\}.$
A state $\varphi$  on $\A$ is called a 
{\it ground state}\/ for $\alpha$ 
if the inequality
\begin{equation}
|\varphi(Y \alpha_z(X)) | \le \| X \| \| Y \| \qquad X \in \A_a, \,\, Y \in \A
\end{equation}
holds for  all $z\in \mathbb{C}_+$ (cf. \cite{BR}).
The generator 
$\delta_\alpha$ of $\alpha$ is defined by
\begin{equation*}
\delta_\alpha(X) = \lim_{t \to 0} \frac{\alpha_t(X) - X}{t}  
\end{equation*}
for $X $ in the domain $D(\delta_\alpha)$  of $\delta_\alpha$.
A ground state $\varphi$ is said to  have {\it spectral gap}\/ $\gamma>0$
if the inequality
\begin{equation}
- i \varphi(X^* \delta_\alpha(X)) \ge \gamma \varphi(X^* X )
\end{equation}
holds for all $ X\in D(\delta_\alpha)$ with $\varphi(X)=0.$

In what follows, we consider the $C^*$-algebra $\OAI$
instead of $\OATI$ to avoid complexity of notation.
We then use the results for $\OAI$ to apply $\OATI$ and $\whatOA = P_N \OATI P_N.$  
\begin{proposition}\label{prop:groundstate}
\begin{enumerate}
\renewcommand{\theenumi}{(\roman{enumi})}
\renewcommand{\labelenumi}{\textup{\theenumi}}
\item
There exists a ground state $\varphi_{A^\infty}$ on $\OAI$ for
$\alpha, \beta, \gamma$-actions.
\item The ground state $\varphi_{A^\infty}$ on $\OAI$ for
$\gamma$-action is unique.
\item For $\gamma$-action on $\OAI$, the ground state 
$\varphi_{A^\infty}$ has spectral gap $1$, that is,
the inequality
\begin{equation*}
- i \varphi_{A^\infty}(X^* \delta_\gamma(X)) \ge \varphi_{A^\infty}(X^* X)
\end{equation*}  
holds for 
$X \in D(\delta_\gamma)$ 
with 
$\varphi_{A^\infty}(X) =0$.
\end{enumerate}
\end{proposition}
\begin{proof}
(i) Let us denote by 
$\FAI$
the fixed point algebra of $\OAI$ under the action $\gamma$. 
There exists a conditional expectation 
$E_{A^\infty}: \OAI\rightarrow \FAI$
defined by
$$
E_{A^\infty}(X) = \int_{\mathbb{T}}\gamma_t(X) dt, \qquad X \in \OAI
$$
where 
$dt$ is the normalized Lebesgue measure on $\mathbb{T}.$
Let us denote by 
$\IAI$ the closed two-sided ideal of $\FAI$ generated by
$S_\mu S_\nu^*$ with $|\mu| = |\nu| \ge 1.$
Since 
$\FAI$ is generated by
$S_\mu S_\nu^*, \, |\mu| = |\nu| \ge 1$
and the unit $1$ of $\OAI$,
any element of $\FAI$ is approximated by finite linear combinations
of elements of  $S_\mu S_\nu^*$ with $|\mu| = |\nu| \ge 1$
and the unit $1$.
Hence $\IAI$ is a maximal ideal of $\FAI$ such that 
$
\FAI/\IAI \cong \mathbb{C}
$
so that the quotient map
$\tau_{A^\infty}: \FAI\rightarrow \mathbb{C}$
is a unital $*$-homomorphism which yields a tracial state on $\FAI.$
Define a state $\varphi_{A^\infty}$ on $\OAI$ by 
$\varphi_{A^\infty}:= \tau_{A^\infty} \circ E_{A^\infty}.$
Note that 
$\varphi_{A^\infty}(S_\mu S_\nu^*) =0$ for 
$\mu, \nu$ with $|\mu| = |\nu| \ge 1,$
in particular
$\varphi_{A^\infty}(S_j S_j^*) =0$ for every $j \in \N,$
so that 
$$
\varphi_{A^\infty}(P_N) 
= \varphi_{A^\infty}(1 - \sum_{j=1}^N S_j S_j^*) 
= \varphi_{A^\infty}(1) = 1
$$
(see also the same state $\varphi_{A^\infty}$ on $(\OAI, \varphi_{A^\infty})=(\TA,\varphi)*(\OI,\varphi_\infty)$
defined in Subsection \ref{subsec.freeproduct}).
Let us denote by $\OalgAI$ 
the $*$-subalgebra of $\OAI$
 algebraically generated by $S_i, i \in \N.$ 
 Any elemnt $X$ of $\OalgAI$ is a finite sum of the form
 \begin{equation}\label{eq:X}
 X= \sum_{|\nu| \ge 1} X_{-\nu}S_\nu^* + X_0 + \sum_{|\mu|\ge 1} S_\mu X_\mu
  \end{equation}
for some
$X_{-\nu}, X_0, X_{\mu} \in \FAI\cap \OalgAI$.
For $z = t + is\in \mathbb{C}$, we have
 \begin{equation*}
\gamma_z(X)
= \sum_{|\nu| \ge 1} X_{-\nu} e^{-iz|\nu|}S_\nu^* 
+ X_0 + 
\sum_{|\mu|\ge 1} e^{iz|\mu|} S_\mu X_\mu.
  \end{equation*}
For $Y \in \OalgAI$ with 
\begin{equation}\label{eq:Y}
 Y= \sum_{|\nu'| \ge 1} Y_{-\nu'}S_{\nu'}^* 
 + Y_0 + \sum_{|\mu'|\ge 1} S_{\mu'} Y_{\mu'},
  \end{equation}
we have
\begin{equation*}
E_{A^\infty}(Y\gamma_z(X))
=  
 \sum_{|\mu'| =|\nu| \ge 1} S_{\mu'} Y_{\mu'} X_{-\nu} e^{-iz|\nu|}S_\nu^* 
+ Y_0 X_0 + 
\sum_{|\nu'| =|\mu|\ge 1} Y_{-\nu'}S_{\nu'}^* e^{iz|\mu|} S_\mu X_\mu
\end{equation*}
so that 
\begin{align*}
\varphi_{A^\infty}(Y \gamma_z(X)) 
=&  \tau_{A^\infty}(E_{A^\infty}(Y \gamma_z(X))) \\
=& \tau_{A^\infty}(Y_0 X_0) 
+ \sum_{|\mu| \ge 1} 
e^{-s |\mu|} e^{i t |\mu|}\tau_{A^\infty}(Y_{-\mu} S_\mu^* S_\mu X_\mu).
\end{align*}
Hence the function
$z \in {\mathbb{C}}_+ \rightarrow \varphi_{A^\infty}(Y \gamma_z(X))$
is bounded in the upper half plane ${\mathbb{C}}_+$.
By Phragmen--Lindel\"{o}f theorem, the supremum of the function
in the region is given on the real line, that is,
$$
|\varphi_{A^\infty}(Y \gamma_z(X))| 
\le \sup_{t \in \R} | \varphi_{A^\infty}(Y \gamma_t(X))| 
\le \sup_{t \in \R} \| Y \| \|  \gamma_t(X)) \| 
\le \| Y \| \| X \|,
$$
showing that $\varphi_{A^\infty}$
 is a ground state for $\gamma$-action on $\OAI$.

We note that the state  $\varphi_{A^\infty}$ on $\OAI$
 is a ground state for $\alpha$-action and similarly for $\beta$-action on $\OAI$.
We in fact see the following.
Put for a word $\mu =(\mu_1,\dots,\mu_m)$ of $\N$
$$
| \mu |_N := |\{ \mu_j \mid 1 \le \mu_j \le N \}|, \qquad 
| \mu |^N := |\{ \mu_j \mid  \mu_j > N \}|. 
$$
For 
$X \in \OalgAI$ of the form \eqref{eq:X}
and
$Y \in \OalgAI$ 
of the form \eqref{eq:Y},
 similarly to $\gamma$-action we have that 
\begin{align*}
\varphi_{A^\infty}(Y \alpha_z(X)) 
= & \tau_{A^\infty}(Y_0 X_0) 
+ \sum_{|\mu| \ge 1} 
e^{-s |\mu|_N} e^{i t |\mu|_N}\tau_{A^\infty}(Y_{-\mu} S_\mu^* S_\mu X_\mu), \\
\varphi_{A^\infty}(Y \beta_z(X)) 
= & \tau_{A^\infty}(Y_0 X_0) 
+ \sum_{|\mu| \ge 1} 
e^{-s |\mu|^N} e^{i t |\mu|^N}\tau_{A^\infty}(Y_{-\mu} S_\mu^* S_\mu X_\mu)
\end{align*}
for $z = t + i s \in {\mathbb{C}}_+$.
Hence 
the functions
$z \in {\mathbb{C}}_+ \rightarrow 
\varphi_{A^\infty}(Y \alpha_z(X)), \varphi_{A^\infty}(Y \beta_z(X))$
are both bounded in ${\mathbb{C}}_+$.
Similarly to the $\gamma$-action, 
 Phragmen--Lindel\"{o}f theorem syas the inequalities
$$
|\varphi_{A^\infty}(Y \alpha_z(X))|,\,
|\varphi_{A^\infty}(Y \beta_z(X))| 
\le \| y \| \| x \|,
$$
showing that $\varphi_{A^\infty}$
 is a ground state for both $\alpha$-action and $\beta$-action on $\OAI$.

(ii) We will show the uniqueness of the ground state for $\gamma$-action on $\OAI$.
Let $\phi$ be a ground state on $\OAI$ for $\gamma$-action.
For $ z = t + i s \in {\mathbb{C}}_+$ and $k \in \N$, 
we have
$$
\phi(S_k \gamma_z(S_k^*)) = e^{-it} e^s \phi(S_k S_k^*)
$$
so that the function
$
{\mathbb{C}}_+ \ni z =t + is  \rightarrow 
|\phi(S_k \gamma_z(S_k^*))|
 = e^s \phi(S_k S_k^*)
$
is not bounded unless $\phi(S_k S_k^*) =0.$
Hence we have
$\phi(S_k S_k^*) =0,$
so that 
$
\phi(P_N) = \phi( 1 - \sum_{j=1}^N S_j S_j^*) = \phi(1) = 1.
$
For words 
$\mu = (\mu_1, \dots, \mu_m), \nu =(\nu_1, \dots,\nu_n)$ 
with $m\ge 1$ or $ n\ge 1$,
suppose that 
$m \ge 1$ and put $\bar{\mu} = (\mu_2,\dots,\mu_m)$ so that 
$\mu = \mu_1 \bar{\mu}$.
By the  Cauchy--Schwartz inequality, we have
\begin{equation*}
| \phi(S_\mu S_\nu^*) |^2 
\le
 \phi(S_{\mu_1} S_{\mu_1}^*)  
 \phi(S_\nu S_{\bar{\mu}}^*  S_{\bar{\mu}} S_\nu^*) =0.
\end{equation*}
Hence we have 
$\phi(S_\mu S_\nu^*) =0$ for $|\mu| \ge 1$ 
and similarly for $| \nu|\ge 1$.
For $X \in \OalgAI$ of the form \eqref{eq:X},
we obtain that 
$$
\phi(X_{-\nu} S_\nu^*) = \phi(S_\mu X_\mu) =0.
$$
As $\IAI$ is the ideal of $\FAI$ generated by
elements of the form $S_\mu S_\nu^*$ with $|\mu | = |\nu| \ge 1$, we have
$\phi(X) =0$ for all $X \in \IAI$.
Hence $\phi$ defines a state on $\FAI/ \IAI \cong \mathbb{C}$
so that $\phi(X) =\tau_{A^\infty}\circ E_{A^\infty}(X) =\varphi_{A^\infty}(X)$
for all $X \in \OAI$.

(iii) We will prove that $\varphi_{A^\infty}$ has spectral gap $1$.
For $X \in \OalgAI$ of the form \eqref{eq:X},
suppose that $\varphi_{A^\infty}(X) =0.$
As
 \begin{equation*}
\varphi_{A^\infty}(X)
= \sum_{|\nu| \ge 1} \varphi_{A^\infty}(X_{-\nu}S_\nu^*)
 + \varphi_{A^\infty}(X_0) 
 + \sum_{|\mu|\ge 1} \varphi_{A^\infty}(S_\mu X_\mu)
  \end{equation*}
and
$\varphi_{A^\infty}(X_{-\nu}S_\nu^*) = \varphi_{A^\infty}(S_\mu X_\mu) = 0,$
we have
$\varphi_{A^\infty}(X_0) =0.$ 
We then have
\begin{equation*}
E_{A^\infty}(X^* X)
= \sum_{|\nu'| =|\nu| \ge 1} S_{\nu'} X_{-\nu'}^* X_{-\nu}S_\nu^* 
+ X_0^* X_0 + 
\sum_{|\mu'| =|\mu|\ge 1} X_{\mu'}^* S_{\mu'}^* S_\mu X_\mu
\end{equation*}
and hence
\begin{equation*}
\varphi_{A^\infty}(X^* X)
= \sum_{|\nu'| =|\nu| \ge 1} \varphi_{A^\infty}(S_{\nu'} X_{-\nu'}^* X_{-\nu}S_\nu^*) 
+ \varphi_{A^\infty}(X_0^* X_0) + 
\sum_{|\mu'| =|\mu|\ge 1} \varphi_{A^\infty}(X_{\mu'}^* S_{\mu'}^* S_\mu X_\mu).
\end{equation*}
Since
$
\varphi_{A^\infty}(X_0^* X_0) 
=\tau_{A^\infty}(X_0^* X_0)
= \tau_{A^\infty}(X_0^*) \tau_{A^\infty}( X_0) 
= \varphi_{A^\infty}(X_0^*) \varphi_{A^\infty}( X_0) 
=0,
$ 
we have
\begin{equation*} 
\varphi_{A^\infty}(X^* X)
= 
\sum_{|\mu|\ge 1} \tau_{A^\infty}(X_{\mu}^* S_{\mu}^* S_\mu X_\mu).
\end{equation*}
On the other hand, for the $\gamma$-action, 
we have
\begin{equation*}
\delta_\gamma(X_{-\nu} S_\nu^*) 
=  - i |\nu| X_{-\nu}S_\nu^*, \qquad
\delta_\gamma(X_0) 
= 0, \qquad
\delta_\gamma(S_\mu X_\mu) 
=  i |\mu| S_\mu X_\mu, 
\end{equation*}
so that 
\begin{equation*}
\delta_\gamma(X)
= - i \sum_{|\nu| \ge 1} |\nu| X_{-\nu}S_\nu^*  + i \sum_{|\mu|\ge 1} |\mu | S_\mu X_\mu.
\end{equation*}
We also have
\begin{align*}
E_{A^\infty}(X^* \delta_\gamma(X))
= & 
\sum_{|\mu'| =|\mu|\ge 1} X_{\mu'}^* S_{\mu'}^* \cdot (i |\mu|  S_\mu X_\mu)
+
\sum_{|\nu'| =|\nu| \ge 1} S_{\nu'} X_{-\nu'}^* \cdot (- i |\nu| X_{-\nu}S_\nu^*) \\
= & 
i \sum_{|\mu|\ge 1} 
|\mu| X_{\mu}^* S_{\mu}^* S_\mu X_\mu
- i
\sum_{|\nu'| =|\nu| \ge 1} |\nu| S_{\nu'} X_{-\nu'}^*  X_{-\nu}S_\nu^*.
 \end{align*}
Since
$\varphi_{A^\infty}( S_{\nu'} X_{-\nu'}^*  X_{-\nu}S_\nu^*) =0$
 for $|\nu'| =|\nu| \ge 1$,
we have 
\begin{equation*} 
- i \varphi_{A^\infty}(X^* \delta_\gamma(X))
=  
 \sum_{|\mu|\ge 1} 
|\mu| \tau_{A^\infty}(X_{\mu}^* S_{\mu}^* S_\mu X_\mu) 
\ge   
 \sum_{|\mu|\ge 1} 
  \tau_{A^\infty}(X_{\mu}^* S_{\mu}^* S_\mu X_\mu) 
=  \varphi_{A^\infty}(X^*X).
 \end{equation*}
Put $\| X \|_1 := \| X \| + \| \delta_\gamma(X) \|$ for $X \in D(\delta_\gamma)$.
As $\OalgAI$ is dense in $D(\delta_\gamma)$ 
under the norm $\| \,\,\, \|_1$,
for any $X \in D(\delta_\gamma)$ with $\varphi_{A^\infty}(X) =0$,
there exists $X_n \in \OalgAI$ such that 
$\lim_{n \to \infty}\| X - X_n \|_1 =0.$
Put $Y_n := X_n - \varphi_{A^\infty}(X_n)$ 
so that $Y_n \in \OalgAI$ and $\varphi_{A^\infty}(Y_n) =0$.
We thus obtain the inequalities
\begin{equation*} 
- i \varphi_{A^\infty}(Y_n^* \delta_\gamma(Y_n))
\ge   
   \varphi_{A^\infty}(Y_n^* Y_n)
 \end{equation*}
Since 
$
| \varphi_{A^\infty}(X_n)| 
=| \varphi_{A^\infty}(X_n) - \varphi_{A^\infty}(X)|
\le \| X_n - X \|
\le \| X - X_n \|_1,
$
we have
$\lim_{n \to \infty}| \varphi_{A^\infty}(X_n)| =0.$ 
Also we have $\delta_\gamma(Y_n) = \delta_\gamma(X_n)$,
so that we have
$$
\lim_{n \to \infty} \| Y_n^* \delta_\gamma(Y_n) - X^* \delta_\gamma(X) \| =0, 
\qquad 
\lim_{n \to \infty} \| Y_n^* Y_n - X^* X \| =0. 
$$
We thus conclude that 
$
- i \varphi_{A^\infty}(X^* \delta_\gamma(X))
\ge   
   \varphi_{A^\infty}(X^* X)
$
for all
$ X \in D(\delta_\gamma).$
 \end{proof}
Therefore we have the following proposition.
\begin{proposition}\label{prop:groundstate2}
\begin{enumerate}
\renewcommand{\theenumi}{(\roman{enumi})}
\renewcommand{\labelenumi}{\textup{\theenumi}}
\item
The ground state $\varphi_{A^\infty}$ on $\OAI$ gives rise to 
a ground state $\varphi_{A^\infty}|_{P_N\OAI P_N}$ 
on $P_N\OAI P_N$ by restriction for $\alpha, \beta, \gamma$-actions,
 respectively.
\item The ground state on $P_N\OAI P_N$ for
$\gamma$-action is unique.
\item  The ground state 
$\varphi_{A^\infty}|_{P_N\OAI P_N}$ on $P_N\OAI P_N$ for $\gamma$-action 
has spectral gap $1$.
\end{enumerate}
\end{proposition}
\begin{proof}
(i)
Since $\varphi_{A^\infty}(P_N) =1$,
the state $\varphi_{A^\infty}$ on $\OAI$ induces a state on 
$P_N \OAI P_N$.
It is straightforward to see that 
the restriction 
$\varphi_{A^\infty}|_{P_N\OAI P_N}$
gives rise to a ground state for $\alpha, \beta, \gamma$-actions 
on $P_N\OAI P_N$, respectively. 

(ii)
Let $\phi$ be a ground stste on $P_N\OAI P_N$ for $\gamma$-action.
The state $\phi$ is extended to a state $\tilde{\phi}$ on $\OAI$
by $\tilde{\phi}(X) = \phi(P_N X P_N)$ for $X \in \OAI$.
The extended state on $\OAI$ 
gives rise to a ground state for $\gamma$-action on $\OAI$.
By the uniqueness of the ground state for $\gamma$-action on $\OAI$,
we have $\phi = \varphi_{A^\infty}.$

(iii)
Since the inequality
$
- i \varphi_{A^\infty}(X^* \delta_\gamma(X)) \ge \varphi_{A^\infty}(X^* X)
$
holds for $X \in D(\delta_\gamma |_{P_N \OAI P_N})$ with 
$\varphi_{A^\infty}(X) =0$,
the assertion is obvious.
\end{proof}

\begin{corollary}\label{cor:groundstate2}
\begin{enumerate}
\renewcommand{\theenumi}{(\roman{enumi})}
\renewcommand{\labelenumi}{\textup{\theenumi}}
\item
Theere are ground state $\varphi_{\widehat{A}}$ on $\whatOA$
 for $\hat{\alpha}^A, \hat{\beta}^A, \hat{\gamma}^A$-actions, respectively.
\item The ground state $\varphi_{\widehat{A}}$ on $\whatOA$ for
$\hat{\gamma}^A$-action is unique.
\item The ground state $\varphi_{\widehat{A}}$ on $\whatOA$ for $\hat{\gamma}^A$-action
has spectral gap $1$.
\end{enumerate}
\end{corollary}

\begin{remark}\label{remark:absenseKMSstate}
Since $\OI$ is a unital subalgebra of $\OAI$,
it is straightforward to see that there is no KMS states on $\OAI$
for every $\alpha, \beta, \gamma$-action,
because there is no KMS states for gauge action on $\OI$ (\cite{OP}).
Similarly, it is easy to see that $\whatOA$ has no KMS states for 
$\hat{\alpha}^A, \hat{\beta}^A, \hat{\gamma}^A$-actions.
\end{remark}


\section{Gauge action in $\pi_1(\Aut(\whatOA))$ }\label{gau}
In this section, we will study 
a relation between the gauge actions $\gamma^A$ on $\OA$ and 
$\hat{\gamma}^A$ on $\whatOA$ 
from the viewpoint of the homotopy classes in the automorphism groups of the $C^*$-algebras.
As in Proposition \ref{prop:homotopybeta} 
the class of the $\hat{\beta}^A$-action in 
$\pi_1(\Aut(\whatOA))$ is trivial,
so the class $[\hat{\gamma}^A]$, 
which is the same class as $[\hat{\alpha}^A]$
 by Corollary \ref{cor:homotopyalpha=gamma},
desseves to study. 
The usual gauge action on  $\OA$ is denoted by $\gamma^A$.
We will prove the following theorem.
\begin{theorem}\label{thm:gaugeinpi1whatOA}
There exists an isomorphism
$\Phi: \pi_1(\Aut(\OA)) \rightarrow \pi_1(\Aut(\whatOA))$
of the fundamental groups such that 
$\Phi([\gamma^A]) =[\hat{\gamma}^A]$,
where
$\hat{\gamma}^A$ denotes the gauge action of the reciprocal algebra $\whatOA$.
\end{theorem}


Some of the notation below follow from \cite{KP}.
Let 
$l_1,\dots, l_N$ 
(resp. $r_1,\dots r_N$) 
be canonical generators of 
$\OA$ (resp. $\OAT$) (i.e., they are partial isometries satisfying the Cuntz--Krieger relations).
Let 
$L^A_1,\dots, L^A_N$ and $R^A_1,\dots, R^A_N$ 
be the left and right creation operators on the Fock space $F_A$
associated to the matrix $A$ as in \cite{KP}.
Let $P_\Omega$ be the projecton of the vacuum vector 
$F_A$ 
satisfying
\[
(L^A_i)^*L^A_i=P_\Omega+\sum_{j=1}^N A(i,j) L^A_j(L^A_j)^*,\quad 
(R^A_i)^*R^A_i=P_\Omega+\sum_{j=1}^N A(j,i) R^A_j(R^A_j)^*,
\]
for $i=1,\dots,N$.
We denote by $\E, \TA, \TAT$ 
the $C^*$-algebras on $F_A$ generated by 
$\{L^A_k, R^A_k\}_{k=1}^N,$
$ \{L^A_k\}_{k=1}^N,$
$ \{R^A_k\}_{k=1}^N$, respectively.
We then have  the following natural extensions (cf. \cite{KP}) 
\begin{gather*}
0 \longrightarrow \calK
\longrightarrow 
\E
\overset{\tilde{\pi}_A}{\longrightarrow}
\OA\otimes\OAT
\longrightarrow 0,\\
0 \longrightarrow \calK
\longrightarrow \TA
\overset{\pi_A}{\longrightarrow}
\OA
\longrightarrow 0,
\qquad
0 \longrightarrow \calK
\longrightarrow \TAT
\overset{\pi_{A^t}}{\longrightarrow}
\OAT
\longrightarrow 0.
\end{gather*}
We identify 
$\TA, \TAT$ (resp. $\OA, \OAT$)  
with the subalgebras in $\E$ (resp. $\OA\otimes \OAT$)  
and one has 
\[
\tilde{\pi}_A(L^A_k)=l_k\otimes 1_\OAT,
\quad 
\tilde{\pi}_A(R^A_k)=1_\OA\otimes r_k,
\qquad k=1,\dots,N.
\]
For a $*$-homomorphism $f:\B \rightarrow \A$ of $C^*$-algebras $\A, \B$,
the mapping cone $C_f(\A)$ is a $C^*$-algebra defined by
\begin{equation*}
C_f(\A) := \operatorname{Cone}(\B\xrightarrow{f}\A) = 
\{(a(t), b) \in (C_0(0,1]\otimes \A)\oplus \B \mid a(1) = f(b) \}.
\end{equation*}
Especially, for the inclusions
$\mathbb{C}\ni \lambda \rightarrow \lambda P_\Omega \in \TAT,\, \,
\calK \rightarrow \TAT, \, \, 
\mathbb{C}\ni\lambda \rightarrow \lambda 1_{\OA}\in \OA
$ and the surjection
$\pi_{A^t}: \TAT \rightarrow \OAT$, 
their  mapping cones are denoted by
$C_{P_\Omega}(\TAT), \, C_{\calK}(\TAT), \, C_{\mathbb{C}}(\OA)$
and $C_{\TAT(\OAT)}$, respectively. 
They mean  
\begin{align*}
C_{P_\Omega}(\TAT) 
:= &  \operatorname{Cone}(\mathbb{C}\rightarrow\TAT) 
=\{a(t)\in C_0(0, 1]\otimes\TAT\; |\; a(1)\in\mathbb{C}P_\Omega\}, \\
C_\calK(\TAT)
 := & \operatorname{Cone}(\calK\rightarrow\TAT)
=\{a(t)\in C_0(0, 1]\otimes\TAT\;|\; a(1)\in\calK\}, \\
C_{\mathbb{C}}(\OA)
:=& \operatorname{Cone}(\mathbb{C}\rightarrow\OA)
=\{ a(t) \in C_0(0,1] \otimes \OA \mid a(1) \in \mathbb{C}1_{\OA}\}, \\
C_{\TAT}(\OAT) 
:=& \operatorname{Cone}(\TAT\to \OAT)
=\{ (a(t), b) \in (C_0(0,1] \otimes \OAT) \oplus \TAT \mid a(1) = \pi_{A^t}(b) \}.
\end{align*}

For separable simple nuclear $C^*$-algebras $\A, \B, \C$,
the Kasparov product
$$
\hat{\otimes}: \sqK(\A,\B) \times \sqK(\B, \C) \longrightarrow \sqK(\A,\C)
$$
is defined in Kasparov $\sqK$-groups $\sqK(\,\,\, , \,\,\,)$ (\cite{Kasparov81}, cf. \cite{Blackadar}).
For a $*$-homomorphism $f: \A\rightarrow \B$, the Kasparov module $(f, \B, 0)$ 
defines an element of $\sqK(\A,\B)$ 
denoted by $\sqK(f)$. 
We write $I_\A = \sqK(\id_\A)$.
For a non-unital $C^*$-algebra $\calC$, we denote by $\calM(\calC)$
the multiplier algebra of $\calC$.

\subsection{Bott projecton, duality class}
We identify the Bott element 
$\beta\in \sqK(\mathbb{C}, C_0(0, 1)\otimes S)(\cong 1\in \Z)$ 
where $S = C_0(0,1)$
with the difference $ q_z - q_1$
of the  projections defined below in the following way:
\begin{gather*}
q_z(t):=
\begin{bmatrix}
(1-t)+tz \\
\sqrt{t(1-t)}(\bar{z}-1)
\end{bmatrix}
[(1-t)+t\bar{z} \quad  \sqrt{t(1-t)}(z-1)], \\
q_1(t):=
\begin{bmatrix}
1&0\\
0&0
\end{bmatrix},
 \quad  t\in [0, 1],
\end{gather*}
where 
$\mathbb{T} = \{ z\in \mathbb{C} \mid |z| =1\}$
and
$z\in C(\mathbb{T})$  
is the unitary 
$z : \mathbb{T}\ni e^{2\pi it}\rightarrow e^{2\pi it}\in\mathbb{C}$ 
generating $\K_1(S)=\Z$.
Considering the cyclic six term exact sequence
\begin{equation*}
\begin{CD}
\K_0(C_0(0,1)\otimes S) @>>> \K_0(C_0(0,1]\otimes S) @>>> \K_0(S) \\
@A{\operatorname{Index}}AA @. @VVV \\
\K_1(S) @<<< 0 @<<< 0,
\end{CD}
\end{equation*}
the Bott element $\beta$ is represented by 
the projections 
$q_z, q_1\in M_2(\mathcal{M}(C_0(0, 1)\otimes S))$ 
such that  
\[
\beta=\operatorname{Index} ([z]_1)=[q_z]_0-[q_1]_0\in \K_0(C_0(0, 1)\otimes S)
\] 
which is also represented as the Cuntz pair (cf. \cite{Blackadar})
\[
\beta=[q_z, q_1]\in \sqK(\mathbb{C}, C_0(0, 1)\otimes S).
\]
Let 
$v:=\sum_{i=1}^Nr_i^*\otimes l_i\in\OAT\otimes \OA$ 
be a partial isometry which satisfies  
\[
\sum_{i, j=1}^N A(i,j)r_ir_i^*\otimes l_jl_j^*=vv^*=v^*v.
\]
The unitary $v_A:=v+(1-vv^*)$ in $\OAT\otimes\OA$ 
defines a $*$-homomorphism 
\[
\delta : S\ni z-1\mapsto v_A-1\in \OAT\otimes \OA,
\]
and the following element is a duality class for 
$\OA$ and $S\OAT$ in the sense of Spanier--Whitehead K-duality (see \cite{KP}):
\[
\delta_A:=\beta\hat{\otimes}(I_{C_0(0, 1)}\otimes \sqK(\delta))
\in \sqK(\mathbb{C}, (C_0(0, 1)\otimes\OAT)\otimes \OA).
\]
One can identify 
$\delta_A$ with the Cuntz pair 
$[q_{v_A}, q_1]\in \sqK(\mathbb{C}, C_0(0, 1)\otimes \OAT\otimes \OA)$ 
where $q_{v_A}, q_1 \in M_2(\calM(C_0(0, 1)\otimes \OAT\otimes \OA))$ 
are defined by
\begin{gather*}
q_{v_A}(t):=
\begin{bmatrix}
(1-t)+tv_A\\
\sqrt{t(1-t)}(v_A^*-1)
\end{bmatrix}
[(1-t)+tv_A^* \quad \sqrt{t(1-t)}(v_A-1)],\\
q_1(t):=
\begin{bmatrix}
1&0\\
0&0
\end{bmatrix}, \quad  t\in [0, 1].
\end{gather*}
%
The inclusion 
$C_{P_\Omega}(\TAT)\rightarrow C_\calK(\TAT)$ of mapping cones
is a $\sqK$-equivalence,
and the diagram
\[
\xymatrix{
C_0(0, 1]\otimes\calK\ar[r] 
& C_\calK(\TAT)\ar[r]^{{\rm id}_{C_0(0, 1]}\otimes(\pi_{A^t})}
& C_0(0, 1)\otimes\OAT\\
& C_{P_\Omega}(\TAT)\ar[u]\ar[ur]^{\varphi}&\\
}
\]
yields a $\sqK$-equivalence 
$\varphi : C_{P_\Omega}(\TAT)\rightarrow  C_0(0, 1)\otimes \OAT$.
Hence the mapping cone 
$C_{P_\Omega}(\TAT)$ and $\OA$ 
are Spanier--Whitehead K-dual via the duality class 
$
\delta_A\hat{\otimes}(\sqK(\varphi)^{-1}\otimes I_\OA)
\in 
\sqK(\mathbb{C}, C_{\calK}(\TAT)\otimes\OA).
$
Let 
\[
V_A
:=(1_{\TAT\otimes\OA}-P_\Omega\otimes 1_\OA
- \sum_{i,j=1}^N A(i,j) R^A_i(R^A_i)^*\otimes l_jl_j^*)
+ \sum_{i=1}^N(R^A_i)^*\otimes l_i
\]
be a coisometory in $\TAT\otimes \OA$ 
which satisfies 
$
V_A^*V_A
=
1_{\TAT\otimes\OA}-P_\Omega\otimes 1_\OA,\; 
V_AV_A^*=1_{\TAT\otimes\OA}$ and $(\pi_{A^t}\otimes {\rm id}_\OA)(V_A)=v_A.
$
\begin{lemma}\label{lem:d1}
The Cuntz pair 
$[Q_{V_A}, Q_I]\in \sqK(\mathbb{C}, C_\calK(\TAT)\otimes\OA)$ 
where 
$Q_{V_A}, Q_I \in M_3(\calM(C_\calK(\TAT)\otimes\OA))$ 
defined below
is a well-defined duality class for 
$C_{P_\Omega}(\TAT)$ and $\OA$:
\begin{gather*}
Q_{V_A}(t)
:=
\begin{bmatrix}
(1-t)I+tV_A\\
\sqrt{t(1-t)}(V_A^*-I)\\
t(P_\Omega\otimes 1_\OA)
\end{bmatrix}
[(1-t)I+tV_A^* \quad 
\sqrt{t(1-t)}(V_A-I) \quad 
t(P_\Omega\otimes 1_\OA)], \\
Q_I(t)
:=\begin{bmatrix}
I & 0 & 0\\
0 & 0 & 0\\
0 & 0 &0
\end{bmatrix},\quad t\in [0, 1],
\end{gather*}
where we write 
$I:=1_{\TAT\otimes\OA}$ for short.
\end{lemma}
\begin{proof}
We first show that $Q_{V_A}(t)$ is a projection satisfying
\[
Q_{V_A}(0)-Q_I(0)=0,\quad Q_{V_A}(1)-Q_I(1)
\in M_3(P_\Omega\otimes\OA)
\]
(i.e., $[Q_{V_A}, Q_I]$ is a well-defined Cuntz pair).
The direct computation yields
\[
[(1-t)I+tV_A^* \quad 
\sqrt{t(1-t)}(V_A-I) \quad 
t(P_\Omega\otimes 1_\OA)]
\begin{bmatrix}
(1-t)I+tV_A\\
\sqrt{t(1-t)}(V_A^*-I)\\
t(P_\Omega\otimes 1_\OA)
\end{bmatrix}
=I
\]
which implies $Q_{V_A}(t)$ is a projection.
As 
$Q_{V_A}(1)-Q_I(1)
=
0\oplus 0\oplus (P_\Omega\otimes 1_\OA)\in M_3(P_\Omega\otimes\OA)$,
the Cuntz pair 
$[Q_{V_A}, Q_I]$ is well-defined.
It is easy to see  
\[[Q_{V_A}, Q_I]\hat{\otimes}(\sqK(\varphi)\otimes I_\OA)
=
[q_{v_A}\oplus 0, q_1\oplus 0]=\delta_A
\] 
so that 
$[Q_{V_A}, Q_I]$ 
is a duality class for 
$C_{P_\Omega}(\TAT)$ and $\OA$ (see also \cite[Remark 3.3]{PennigSogabe}).
\end{proof}
In the rest of this subsection, 
we will see that the above construction also gives another duality class 
$[\tilde{Q}_{V_A}, \tilde{Q}_I]
\in 
\sqK(\mathbb{C}, \TAT\otimes C_{\mathbb{C}}(\OA))$ 
for $\TAT$ and 
$
D(\TAT)
=C_{\mathbb{C}}(\OA).
$
We identify 
$\TAT\otimes C_{\mathbb{C}}(\OA)$ 
with 
$
\{a(t)\in C_0(0, 1]\otimes\TAT\otimes\OA\;|\; a(1)\in\TAT\otimes \mathbb{C}1_\OA\}.
$
\begin{lemma}
The projections 
$Q_{V_A}, Q_I\in M_3(C[0, 1]\otimes\TAT\otimes\OA)$ 
lie in 
$M_3(\mathcal{M}(\TAT\otimes C_{\mathbb{C}}(\OA)))$,
and they define a well-defined Cuntz pair 
$[\tilde{Q}_{V_A}, \tilde{Q}_I]
\in \sqK(\mathbb{C}, \TAT\otimes C_{\mathbb{C}}(\OA)).$
\end{lemma}
\begin{proof}
Since 
$
Q_{V_A}(1)
=
1_{\TAT\otimes \OA}\oplus 0\oplus (P_\Omega\otimes \mathbb{C}1_\OA)
\in M_3(\TAT\otimes 1_\OA)$,
one has 
\[
\tilde{Q}_{V_A}
:=Q_{V_A}\in M_3(\mathcal{M}(\TAT\otimes C_{\mathbb{C}}(\OA))).
\]
We also write 
$\tilde{Q}_I
:=Q_I\in M_3(\mathcal{M}(\TAT\otimes C_{\mathbb{C}}(\OA)))$.
One has 
$\tilde{Q}_{V_A}(0)-\tilde{Q}(0)=0$ 
and 
$\tilde{Q}_{V_A}(1)-\tilde{Q}_I(1)
= 0\oplus 0\oplus (P_\Omega\otimes 1_\OA)
\in M_3(\TAT\otimes 1_\OA)$ 
which imply 
$\tilde{Q}_{V_A}-\tilde{Q}_I
\in\TAT\otimes C_{\mathbb{C}}(\OA)$.
Thus, 
the Cuntz pair 
$[\tilde{Q}_{V_A}, \tilde{Q}_I]
\in \sqK(\mathbb{C}, \TAT\otimes C_{\mathbb{C}}(\OA))$ 
is well-defined.
\end{proof}
We will prove the following theorem.
\begin{theorem}\label{thm:dclass}
For any separable $C^*$-algebras $\calP, \calR$,
the natural map 
$[\tilde{Q}_{V_A}, \tilde{Q}_I]\hat{\otimes}-$ 
defined by 
\[
\sqK(\calP\otimes\TAT, \calR) \ni x \mapsto 
(I_\calP\otimes [\tilde{Q}_{V_A}, \tilde{Q}_I])
\hat{\otimes} (x\otimes I_{C_{\mathbb{C}}(\OA)})
\in \sqK(\calP, \calR\otimes C_{\mathbb{C}}(\OA))
\]
is an isomorphism.
In particular,
the element 
$[\tilde{Q}_{V_A}, \tilde{Q}_I]
\in  
\sqK(\mathbb{C}, \TAT\otimes C_{\mathbb{C}}(\OA))$ 
is a duality class for $\TAT$ and $C_{\mathbb{C}}(\OA)$.
\end{theorem}
To prove Theorem \ref{thm:dclass}, 
we provide several notation and lemmas.
For an extension 
$0 \to \calK\to \E\xrightarrow{\pi} \A\to 0$ 
of nuclear $C^*$-algebras $\E, \A$,
we denote by $C_\E(\A)$
the mapping cone $C_{\pi}(\A)$
for the surjection $\pi: \E\rightarrow \A$.
There are natural homomorphisms
\begin{equation*}
i(\pi): S\A \ni a(t) \to (a(t), 0) \in C_{\E}(\A), \quad 
e(\pi) : C_\E(\A) \ni (a(t), x) \to x \in \E
\end{equation*}
which make the following diagram commutative
\begin{equation*}
\begin{CD}
0 @>>> S\A @>{i(\pi)}>> C_{\E}(\A) @>{e(\pi)}>> \E @>>> 0 \\
@. @|  @VVV  @VV{\pi}V  @. \\
0 @>>> S\A @>>> C_0(0, 1]\otimes \A @>>> \A @>>> 0.
\end{CD}
\end{equation*}
Consider the homomorphism
$
j(\E) : \calK\ni x\mapsto (0, x)
\in C_\E(\A).
$
It is well-known that 
$\sqK(j(\E))$ is a $\sqK$-equivalence (cf. \cite{Blackadar}).
Let $\eta(\E)$ be the element of $\sqK(S\A, \calK)$ defined by
\[
\eta(\E) 
:=\sqK(i(\pi))
\hat{\otimes}\sqK(j(\E))^{-1}\in \sqK(S\A, \calK).
\]
We can identify the extension sequence 
$
S\A\xrightarrow{\eta(\E)}\calK\to \E\to \A$ 
with the mapping cone sequence 
$S\A\to C_{\E}(\A)\to \E\to \A$ 
(cf. \cite{PennigSogabe}).
For two $C^*$-algebras $\calC, \calD$,
we write 
\[
\sigma_{\calC, \calD} : \calC\otimes \calD\ni c\otimes d\mapsto d\otimes c
\in \calD\otimes \calC.
\]
We write homomorphisms
\begin{gather*}
u_\OA : \mathbb{C}\ni \lambda\mapsto \lambda 1_\OA\in\OA,
\qquad 
P_\Omega : \mathbb{C}\ni\lambda\mapsto\lambda P_\Omega\in \calK,\\
e_A : C_{\mathbb{C}}(\OA)\ni (a(t), x)\mapsto x\in\mathbb{C},\\
i_A : S\OA\ni a(t)\mapsto (a(t), 0)\in C_{\mathbb{C}}(\OA), \\
T_A : S\OAT\ni a(t)\mapsto(a(t), 0)\in C_{\TAT}(\OAT).
\end{gather*}
\begin{lemma}\label{lem:c1}
The following diagrams commute and all vertical maps are isomorphisms:
\begin{equation*}
\begin{CD}
\sqK(\calP\otimes \calK, \calR) @>{\eta(\TAT)\hat{\otimes}-}>> \sqK(\calP\otimes S\OAT, \calR) \\
@V{P_\Omega \hat{\otimes} -}VV @VV{(-\delta_A)\hat{\otimes}-}V \\
\sqK(\calP, \calR) @>{-\hat{\otimes}u_\OA}>> \sqK(\calP, \calR\otimes \OA)
\end{CD}
\end{equation*}
\begin{equation*}
\begin{CD}
\sqK(\calP\otimes S\calK, \calR)
@>{(\beta\otimes I_\OAT)\hat{\otimes}(I_S\otimes\eta(\TAT))\hat{\otimes}-}>>
\sqK(\calP\otimes \OAT, \calR)  \\
@V{
\beta\hat{\otimes}(I_{S^2}\otimes P_\Omega)\hat{\otimes}(I_S\otimes \sigma_{S, \calK})\hat{\otimes}-
}VV 
@VV{
\delta_A\hat{\otimes}(\sigma_{S, \OAT}\otimes I_\OA)\hat{\otimes}-
}V \\
\sqK(\calP, \calR\otimes S)
@>{-\hat{\otimes}(I_S \otimes u_{\OA}}>>
\sqK(\calP, \calR\otimes S\OA).
\end{CD}
\end{equation*}
\end{lemma} 
\begin{proof}
Note that 
$(S\OAT, \OA), (\calK, \mathbb{C}), (\OAT, S\OA), (S\calK, S)$ 
are pairs of $C^*$-algebras of Spanier--Whitehead K-dual 
and all vertical maps are defined by the Kasparov products
 via their duality classes,
so they are isomorphisms.
We first check the commutativity of the upper diagram.
For $f\in \sqK(\calP\otimes \calK, \calR)$,
one has 
\begin{gather*}
P_\Omega\hat{\otimes} f\hat{\otimes}u_\OA
=
(I_\calP\otimes (P_\Omega\otimes u_\OA))\hat{\otimes}(f\otimes I_\OA),\\
-\delta_A\hat{\otimes} (\eta(\TAT)\hat{\otimes} f)
=
(I_\calP\otimes (-\delta_A\hat{\otimes}(\eta(\TAT)\otimes I_\OA)))\hat{\otimes}(f\otimes I_\OA),
\end{gather*}
where we write 
$P_\Omega\otimes u_\OA : \mathbb{C}\ni \lambda\mapsto \lambda P_\Omega\otimes 1_\OA\in \calK\otimes\OA$ 
by abuse of notation 
and will identify this with the Cuntz pair 
$[P_\Omega\otimes 1_\OA, 0]\in \sqK(\mathbb{C}, \calK\otimes\OA)$.
Thus, it is enough to show 
\[
[P_\Omega\otimes 1_\OA, 0]+\delta_A\hat{\otimes}(\eta(\TAT)\otimes I_\OA)=0 
\text{ in } \sqK(\mathbb{C}, \calK\otimes\OA),\]
equivalently,
\[
[P_\Omega\otimes 1_\OA, 0]\hat{\otimes}(\sqK(j(\TAT))\otimes I_\OA)
+
\delta_A\hat{\otimes}(\sqK(T_A)\otimes I_\OA)=0
\text{ in } 
\sqK(\mathbb{C}, C_{\TAT}(\OAT)\otimes \OA).
\]
Recall the Cuntz pair $\delta_A=[q_{v_A}, q_1]$.
Considering the map between the following unitizations 
\[\mathbb{C}1_{C[0, 1]\otimes \OAT}+S\OAT
\to 
\mathbb{C}(1_{C[0, 1]\otimes\OAT}, 1_\TAT)
+
C_{\TAT}(\OAT),
\]
one can identify 
$\delta_A\hat{\otimes}(\sqK(T_A)\otimes I_\OA)$ 
with the  Cuntz pair 
\[
[(q_{v_A}(t)\oplus 0, Q_I(1)), (q_1(t)\oplus 0, Q_I(1))]
\] 
defined by the maps
\begin{align*}
\mathbb{C}\ni \lambda\mapsto \lambda (q_{v_A}(t)\oplus 0, Q_I(1))
& \in M_3(\mathcal{M}(C_{\TAT}(\OAT) \otimes \OA)),\\
\mathbb{C}\ni \lambda\mapsto \lambda (q_1(t)\oplus 0, Q_I(1))
& \in M_3(\mathcal{M}(C_{\TAT}(\OAT) \otimes \OA))
\end{align*}
where $C_{\TAT}(\OAT) \otimes \OA$ is naturally identified with the mapping cone 
$\operatorname{Cone}(\TAT\otimes\OA\to\OAT\otimes\OA).$ 
Since 
$[P_\Omega\otimes 1_\OA, 0]\hat{\otimes}(\sqK(j(\TAT))\otimes I_\OA)$ 
is identified with the Cuntz pair 
$
[(0, P_\Omega\otimes 1_\OA), (0, 0)]
$ 
of the maps
\begin{align*}
\mathbb{C}\ni\lambda\mapsto \lambda (0, P_\Omega\otimes 1_\OA)
& \in C_{\TAT}(\OAT) \otimes \OA,\\
\mathbb{C}\ni\lambda \mapsto (0, 0)
\in & C_{\TAT}(\OAT) \otimes \OA.
\end{align*}
The direct computation yields
\begin{align*}
 &[P_\Omega\otimes 1_\OA, 0]\hat{\otimes}(\sqK(j(\TAT))\otimes I_\OA)
  +\delta_A\hat{\otimes}(\sqK(T_A)\otimes I_\OA) \\
=&[(0, P_\Omega\otimes 1_\OA), (0,0)]+[(q_{v_A}(t)\oplus 0, Q_I(1)), (q_1(t)\oplus 0, Q_I(1))]\\
=&[(q_{v_A}(t)\oplus 0, Q_I(1))\oplus (0, P_\Omega\otimes 1_\OA), \; (q_1(t)\oplus 0, Q_I(1))\oplus (0, 0)]\\
\sim_h&[(q_{v_A}(t)\oplus 0, Q_{V_A}(1))\oplus (0, 0), (q_1(t)\oplus 0, Q_I(1))\oplus (0, 0)]\\
\sim_h&[(q_{v_A}(st)\oplus 0, Q_{V_A}(s)), (q_1(st)\oplus 0, Q_I(s))]\\
\sim_h&[(q_{v_A}(0)\oplus 0, Q_{V_A}(0)), (q_1(0)\oplus 0, Q_I(0))]\\
=&[(q_1(0)\oplus 0, Q_I(0)), (q_1(0)\oplus 0, Q_I(0))]\\
=&0,
\end{align*}
where $\sim_h$ means homotopy equivalence, which gives rise to the same element in $\sqK(\,\,\,, \,\,\,)$.
We next check the commutativity of the lower diagram. 
It is enough to show
\begin{align*}
 &\beta\hat{\otimes}(I_{S^2}\otimes \sqK(P_\Omega))
       \hat{\otimes}(I_S\otimes \sqK(\sigma_{S,\calK}))\hat{\otimes}(I_{S \calK\otimes S}\otimes \sqK(u_\OA))\\
=&\delta_A\hat{\otimes}(\sqK(\sigma_{S, \OAT})\otimes I_\OA)\hat{\otimes}(\beta\otimes I_{\OAT\otimes S\otimes\OA})\hat{\otimes}(I_S\otimes \eta(\TAT)\otimes I_{S\OA}).
\end{align*}
In the previous argument,
we proved 
\[
-\delta_A\hat{\otimes}(\eta(\TAT)\otimes I_\OA)
=\sqK(\calP_\Omega\otimes u_\OA)=\sqK(P_\Omega)\hat{\otimes}(I_\calK\otimes \sqK(u_\OA))
\]
and one has 
\[
\sqK(\sigma_{S, S})=-I_{S^2},
\quad 
\sqK(\sigma_{S^2, S})=(-1)^2I_{S^3}, 
\quad 
\sqK(\sigma_{S, S \calK})=-(I_S\otimes \sqK(\sigma_{S, \calK})).
\]
Hence the direct computation yields
\begin{align*}
 &\delta_A\hat{\otimes}(\sqK(\sigma_{S, \OAT})\otimes I_\OA)
  \hat{\otimes}(\beta\otimes I_{\OAT\otimes S\otimes \OA})
  \hat{\otimes}(I_S\otimes \eta(\TAT)\otimes I_{S\OA})\\
=&\beta\hat{\otimes}(I_{S^2}\otimes \delta_A)
  \hat{\otimes}(I_{S^2}\otimes \sqK(\sigma_{S, \OAT})\otimes I_\OA)
  \hat{\otimes}(I_S\otimes\eta(\TAT)\otimes I_{S\OA})\\
=&\beta\hat{\otimes}(I_{S^2}\otimes \delta_A)
  \hat{\otimes}(\sqK(\sigma_{S^2, S})\otimes I_{\OAT\otimes\OA})
  \hat{\otimes}(I_{S^2}\otimes\eta(\TAT)\otimes I_\OA)
  \hat{\otimes}(\sqK(\sigma_{S, S \calK})\otimes I_\OA)\\
=&(-1)^2\beta\hat{\otimes}(I_{S^2}\otimes\delta_A)
  \hat{\otimes}(I_{S^2}\otimes\eta(\TAT)\otimes I_\OA)
  \hat{\otimes}(\sqK(\sigma_{S, S\calK})\otimes I_\OA)\\
=&-\beta\hat{\otimes}(I_{S^2}\otimes (\sqK(P_\Omega)
  \hat{\otimes}(I_\calK\otimes \sqK(u_\OA))))
  \hat{\otimes}(\sqK(\sigma_{S, S\calK})\otimes I_\OA)\\
=&-\beta\hat{\otimes}(I_{S^2}\otimes \sqK(P_\Omega))
  \hat{\otimes}\sqK(\sigma_{S, S\calK})
  \hat{\otimes}(I_{S\calK\otimes S}\otimes \sqK(u_\OA))\\
=&\beta\hat{\otimes}(I_{S^2}\otimes \sqK(P_\Omega))
  \hat{\otimes}(I_S\otimes \sqK(\sigma_{S, \calK}))
  \hat{\otimes}(I_{S\calK\otimes S}\otimes \sqK(u_\OA)).
\end{align*}
\end{proof}
\begin{lemma}\label{c2}
The following diagrams commute:
\begin{equation*}
\begin{CD}
\sqK(\calP\otimes\OAT, \calR) @>{(\TAT\to\OAT)\hat{\otimes}-}>>
 \sqK(\calP\otimes\TAT, \calR)
 @>{(\calK\to\TAT)\hat{\otimes}-}>> \sqK(\calP\otimes\calK, \calR)
  \\
@V{\delta_A\hat{\otimes}(\sigma_{S, \OAT}\otimes I_\OA)\hat{\otimes}-}VV 
  @V{
  [\tilde{Q}_{V_A}, \tilde{Q}_I]\hat{\otimes}-
  }VV
@V{
P_\Omega\hat{\otimes}-
}VV \\
\sqK(\calP, \calR\otimes S\OA) @>>{-\hat{\otimes}i_A}> 
\sqK(\calP, \calR\otimes C_\mathbb{C}(\OA)) @>>{-\hat{\otimes}e_A}> 
\sqK(\calP, \calR).
\end{CD}
\end{equation*}
\end{lemma}
\begin{proof}
We first show that the right square commutes.
For $f\in \sqK(\calP\otimes\TAT, \calR)$,
one has
\begin{gather*}
\sqK(P_\Omega)\hat{\otimes}\sqK(\calK\to\TAT)\hat{\otimes} f
=(I_\calP\otimes \sqK(\mathbb{C}\xrightarrow{P_\Omega}\TAT))\hat{\otimes} f,\\
[\tilde{Q}_{V_A}, \tilde{Q}_I]\hat{\otimes} f \hat{\otimes} \sqK(e_A)
=(I_\calP\otimes ([\tilde{Q}_{V_A}, \tilde{Q}_I]\hat{\otimes}(I_\TAT\otimes \sqK(e_A))))\hat{\otimes} f,
\end{gather*}
where we write $P_\Omega : \mathbb{C}\ni\lambda\mapsto\lambda P_\Omega\in\TAT$ by abuse of notation.
The direct computation yields
\begin{align*}
 & [\tilde{Q}_{V_A}, \tilde{Q}_I]\hat{\otimes}(I_\TAT\otimes \sqK(e_A))\\
=&[Q_{V_A}(1), Q_I(1)]
=[1\oplus 0\oplus P_\Omega, 1\oplus 0\oplus 0]
=\sqK(\mathbb{C}\xrightarrow{P_\Omega}\TAT),
\end{align*}
which implies 
\[
(I_\calP\otimes \sqK(\mathbb{C}\xrightarrow{P_\Omega}\TAT))\hat{\otimes} f
=(I_\calP\otimes ([\tilde{Q}_{V_A}, \tilde{Q}_I]\hat{\otimes}(I_\TAT\otimes \sqK(e_A))))\hat{\otimes} f.
\]
We next show the left square commutes.
It is enough to show
\begin{align*}
&[\tilde{Q}_{V_A}, \tilde{Q}_I]
 \hat{\otimes}(\sqK(\TAT\to\OAT)\otimes I_{C_{\mathbb{C}}(\OA)})\\
=
&\delta_A\hat{\otimes}(\sqK(\sigma_{S, \OAT})\otimes I_\OA)
 \hat{\otimes}(I_\OAT\otimes \sqK(i_A)).
\end{align*}
Identifying 
$\OAT\otimes S\otimes\OA$ 
with 
$C_0((0, 1), \OAT\otimes\OA)$,
the composition 
$\delta_A\hat{\otimes}(\sqK(\sigma_{S, \OAT})\otimes I_\OA)$ 
is the Cuntz pair 
$[q_{v_A}, q_1]\in \sqK(\mathbb{C}, C_0((0,1), \OAT\otimes\OA)).$
Since 
\[
\OAT\otimes C_{\mathbb{C}}(\OA)
=\{ a(t)\in C_0((0, 1], \OAT\otimes\OA)\;|\; a(1)\in\OAT\otimes 1_\OA\},
\]
the following unitization of  the map 
$
{\rm id}_\OAT\otimes i_A$ 
\[
\mathbb{C}1_{C([0, 1], \OAT\otimes\OA)}+C_0((0, 1), \OAT\otimes\OA)
\rightarrow 
 \mathbb{C}1_{C([0, 1], \OAT\otimes\OA)}+\OAT\otimes C_{\mathbb{C}}(\OA)
 \]
gives
\begin{align*}
 &\delta_A\hat{\otimes}(\sqK(\sigma_{S, \OAT})\otimes I_\OA)\hat{\otimes}(I_\OAT\otimes \sqK(i_A)) \\
=&[q_{v_A}, q_1]\hat{\otimes}\sqK({\rm id}_\OAT\otimes i_A)
=[q_{v_A}, q_1]
\in \sqK(\mathbb{C}, \OAT\otimes C_{\mathbb{C}}(\OA).
\end{align*}
On the other hands,
we have
\begin{align*}
[\tilde{Q}_{V_A}, \tilde{Q}_I]\hat{\otimes}(\sqK(\TAT\to\OAT)\otimes I_{C_{\mathbb{C}}(\OA)})
=&[q_{v_A}, q_1]\in \sqK(\mathbb{C}, \OAT\otimes C_{\mathbb{C}}(\OA)))
\end{align*}
by definition, and this completes the proof.
\end{proof}
\begin{proof}[{Proof of Theorem \ref{thm:dclass}}]
Since 
\[
S\OAT\xrightarrow{\eta(\TAT)} 
\calK\longrightarrow\TAT\longrightarrow\OAT
\] 
and 
\[
S
\xrightarrow{Su_\OA}
S\OA
\xrightarrow{i_A}
C_{\mathbb{C}}(\OA)
\xrightarrow{e_A}
\mathbb{C}   
\]
are mapping cone sequence,
one has the following exact sequences
\begin{gather*}
\sqK(\calP\otimes S\calK, \calR) 
\rightarrow
\sqK(\calP\otimes \OAT, \calR)
\rightarrow
\sqK(\calP\otimes\TAT, \calR)
\rightarrow
\sqK(\calP\otimes\calK, \calR)
\rightarrow
\sqK(\calP\otimes S\OAT, \calR)
 \\
\sqK(\calP, \calR\otimes S)
\rightarrow
\sqK(\calP, \calR\otimes S\OA) 
 \rightarrow
\sqK(\calP, \calR\otimes C_{\mathbb{C}}(\OA))
\rightarrow
\sqK(\calP, \calR)
\rightarrow
\sqK(\calP, \calR\otimes\OA).
\end{gather*}
Lemma \ref{lem:c1}, \ref{c2} 
together with the five lemma prove that the map 
\[
[\tilde{Q}_{V_A}, \tilde{Q}_I]\hat{\otimes}- 
: \sqK(\calP\otimes \TAT, \calR) \rightarrow 
  \sqK(\calP, \calR\otimes C_{\mathbb{C}}(\OA))\]
is an isomorphism.
Therefore
 \cite[Lemma 3.4]{PS} shows 
 $[\tilde{Q}_{V_A}, \tilde{Q}_I]
 \in \sqK(\mathbb{C}, \TAT\otimes C_{\mathbb{C}}(\OA))$ 
 is a duality class for 
 $\TAT$ and $C_{\mathbb{C}}(\OA)$.
\end{proof}

\subsection{Gauge actions in $\pi_1(\Aut(\OA))$ and 
$\pi_1(\Aut(\whatOA))$}
Recall that the circle group $\{z \in \mathbb{C} \mid |z| =1\}$ 
is denoted by $\mathbb{T}.$
For $\TA$, $\OA$ and $\OATI$,
the gauge actions on their respect algebras are defined by 
\begin{gather*}
(\gamma_\TA)_z : \TA\ni L_i^A\mapsto zL_i^A\in\TA, \quad i=1, \dots, N, \, z\in\mathbb{T},\\
\gamma^A_z : \OA\ni l_i\mapsto zl_i \in\OA, \quad i=1, \dots, N, \, z\in \mathbb{T},\\
(\gATI)_z : \OATI\ni S_i\mapsto zS_i \in\OATI,\quad i=1, 2, \dots, \, z\in\mathbb{T}.
\end{gather*}
Put $P_N = 1 - \sum_{j=1}^N S_j S_j^*$ in $\OATI$.
Since 
$(\gATI)_z(P_N)= P_N$
and 
$\whatOA = P_N \OATI P_N$,
the restriction of the third gauge action above 
defines a gauge action on $\whatOA$
written as $\hat{\gamma}^A$.

The fundamental group 
$\pi_1(\Aut(\A))$ of the automorphism group $\Aut(\A)$ of a unital $C^*$-algebra $\A$
is the set of homotopy classes of the base point preserving continuous maps :
\[
\pi_1(\Aut(\A))
:=\{\alpha : (\mathbb{T}, 1)\to (\Aut(\A), {\rm id}_A)\}/\sim_{\text{homotopy}}.
\]
In \cite{Sogabe2022}, it was proved that 
$\pi_i(\Aut(\A)) \cong \pi_i(\Aut(\widehat{\A})), i=1,2$
for a unital Kirchberg algebra.
In \cite{MatSogabe}, it was also proved that 
the homotopy groups  
$\pi_i(\Aut(\OA)), i=1,2$ for the Cuntz--Krieger algebra is a complete set of 
invariants of the isomorphism class of $\OA$.
The gauge action $\gamma^A$ 
naturally defines an element of the homotopy group $\pi_1(\Aut(\OA))$,
and similarly $\hat{\gamma}^A$ defines an elemsnt of $\pi_1(\Aut(\whatOA))$.

Recall that M. Dadarlat \cite{Dadarlat2007} showed that for a unital Kirchberg algebra $\A$
the map
\begin{equation}\label{eq:Dadarlatpi1}
\pi_1(\Aut(\A))\ni [z\mapsto \alpha_{z}]\mapsto 
[C(\alpha), C(l)]\in \sqK(C_{\mathbb{C}}(\A), SC_0(\mathbb{T}, 1)\otimes \A)
\end{equation}
gives rise to an isomorphism of groups, 
where $C_{\mathbb{C}}(\A)$ is the mapping cone 
$\operatorname{Cone}(\mathbb{C}\to \A) =\{a(t) \in C_0(0,1] \otimes \A \mid a(1) \in \mathbb{C} 1_{\A} \},
C_0(\mathbb{T}, 1):=\{f\in C(\mathbb{T})\;|\; f(1)=0\}$, 
and $C(\alpha), C(l)$ denote the maps
\begin{gather*}
C(\alpha) : C_{\mathbb{C}}(\A)\ni a(t)
\mapsto 
\alpha_z(a(t))\in C[0, 1]\otimes C(\mathbb{T})\otimes \A\subset\mathcal{M}(SC_0(\mathbb{T}, 1)\otimes \A), \\
C(l) : C_{\mathbb{C}}(\A)\ni a(t)
\mapsto 
a(t)\in C[0, 1]\otimes C(\mathbb{T})\otimes \A\subset\mathcal{M}(SC_0(\mathbb{T}, 1)\otimes \A).
\end{gather*}
Since 
$\alpha_1={\rm id}_\A$ and 
$\alpha_z(a(0))=0=a(0), \alpha_z(a(1))=a(1)\in\mathbb{C}1_\A$,
$[C(\alpha), C(l)]$ defines a well-defined Cuntz pair.
%
We provide a lemma to prove Theorem \ref{thm:gaugeinpi1whatOA}.
\begin{lemma}
\[
(\pi_1(\Aut(\whatOA)), [\hat{\gamma}^A])
\cong 
(\sqK(C_{P_\Omega}(\TAT), SC_0(\mathbb{T}, 1)\otimes\TAT), [C(\gamma_\TAT), C(l)]).\]
\end{lemma}
\begin{proof}
For the inclusion map
$P_N: \mathbb{C} \ni \lambda \rightarrow \lambda P_N \in \OATI$,
let us denote by 
$C_{P_N}(\OATI)$
the mapping cone 
$ \operatorname{Cone}(\mathbb{C}\xrightarrow{P_N}\OATI)
= \{ a(t) \in C_0(0,1]\otimes \OATI \mid a(1) \in \mathbb{C} P_N\}.
$
Since 
$(\gamma_{\TAT})_z(P_\Omega)=P_\Omega$ 
and 
$(\gATI)_z(P_N)= P_N$,
we have  well-defined Cuntz pairs 
\begin{gather*}
[C(\gamma_\TAT), C(l)]
\in \sqK(C_{P_\Omega}(\TAT), SC_0(\mathbb{T}, 1)\otimes\TAT),\\
[C(\gATI), C(l)]
\in \sqK(C_{P_N}(\OAI), SC_0(\mathbb{T}, 1)\otimes \OATI).
\end{gather*}
The map 
$\theta : \TAT\ni R^A_i\mapsto S_i\in\OATI$ 
is $\sqK$-equivalent satisfying 
$\theta(P_\Omega)= P_N$.
It induces a $\sqK$-equivalence 
$\Theta : C_{P_\Omega}(\TAT) \to C_{P_N}(\OATI)$.
We then have  
\[
\sqK(\Theta)\hat{\otimes}[C(\gATI), C(l)]
=[C(\gamma_\TAT), C(l)]\hat{\otimes}\sqK(\theta),
\]
which implies 
\begin{align*}
&(\sqK(C_{P_N}(\OATI), SC_0(\mathbb{T}, 1)\otimes \OATI), [C(\gATI), C(l)])\\
\cong
&(\sqK(C_{P_\Omega}(\TAT), SC_0(\mathbb{T}, 1)\otimes\TAT), [C(\gamma_\TAT), C(l)]).
\end{align*}
The same argument for the inclusion map 
$\whatOA\rightarrow\OATI$ 
shows
\begin{align*}
&(\sqK(C_{P_N}(\OATI), SC_0(\mathbb{T}, 1)\otimes \OATI), [C(\gATI), C(l)])\\
\cong
&(\sqK(C_{P_\Omega}(\TAT), SC_0(\mathbb{T}, 1)\otimes \whatOA), [C(\hat{\gamma}^A), C(l)])\\
\cong
& (\pi_1(\Aut(\whatOA)), [\hat{\gamma}^A]).
\end{align*}
\end{proof}
\begin{proof}[{Proof of Theorem \ref{thm:gaugeinpi1whatOA}}]
Note that 
$\pi_1(\Aut(\OA))$ 
is abelian and one has
\[
(\pi_1(\Aut(\OA)), [\gamma^A])
\cong 
(\pi_1(\Aut(\OA)), -[\gamma^A])
=(\pi_1(\Aut(\OA)), [(\gamma^A)^{-1}]).
\]
By the above lemma and M. Dadarlat's formula \eqref{eq:Dadarlatpi1},
it is enough to show
\begin{align*}
&(\sqK(C_{\mathbb{C}}(\OA), SC_0(\mathbb{T}, 1)\otimes \OA), [C((\gamma^A)^{-1}), C(l)])\\
\cong
&(\sqK(C_{P_\Omega}(\TAT), SC_0(\mathbb{T}, 1)\otimes \TAT), [C(\gamma_\TAT), C(l)]).
\end{align*}
By Lemma \ref{lem:d1},
the Kasparov product 
\[
[Q_{V_A}, Q_I]\hat{\otimes}- 
: \sqK(C_{P_\Omega}(\TAT), SC_0(\mathbb{T}, 1)\otimes\TAT)
\to 
\sqK(\mathbb{C}, SC_0(\mathbb{T}, 1)\otimes\TAT\otimes\OA)
\]
gives an isomorphism.
Since 
$(\gamma_\TAT)_z(P_\Omega)=P_\Omega$ 
and 
$((\gamma_\TAT)_z\otimes {\rm id}_\OA)(Q_I(t))=1\oplus 0\oplus 0=(l_z\otimes{\rm id}_\OA)(Q_I(t))$ 
for $z\in \mathbb{T}, t\in [0, 1]$, 
one has 
\[[Q_{V_A}, Q_I]\hat{\otimes}([C(\gamma_\TAT), C(l)]\otimes I_\OA)
=[((\gamma_\TAT)_z\otimes {\rm id}_\OA)(Q_{V_A}(t)), Q_{V_A}(t)]\]
where 
$((\gamma_\TAT)_z\otimes{\rm id}_\OA)(Q_{V_A}(t))$ and $Q_{V_A}(t)$ 
denote the maps
\begin{gather*}
\mathbb{C}\ni\lambda\mapsto 
\lambda((\gamma_\TAT)_z\otimes{\rm id}_\OA)(Q_{V_A}(t))
\in M_3(\mathcal{M}(SC_0(\mathbb{T}, 1)\otimes\TAT\otimes\OA)),\\
\mathbb{C}\ni\lambda\mapsto \lambda Q_{V_A}(t)
\in M_3(\mathcal{M}(SC_0(\mathbb{T}, 1)\otimes\TAT\otimes\OA)),
\end{gather*}
respectively, which satisfy
$((\gamma_\TAT)_z\otimes{\rm id}_\OA)(Q_{V_A}(t))-Q_{V_A}(t)
\in M_3(SC_0(\mathbb{T}, 1)\otimes\TAT\otimes\OA)$.
We thus have
\begin{align*}
&(\sqK(C_{P_\Omega}(\TAT), SC_0(\mathbb{T}, 1)\otimes \TAT), [C(\gamma_\TAT), C(l)])\\
\cong
&(\sqK(\mathbb{C}, SC_0(\mathbb{T}, 1)\otimes\TAT\otimes\OA), [((\gamma_\TAT)_z\otimes {\rm id}_\OA)(Q_{V_A}(t)), Q_{V_A}(t)]).
\end{align*}
By Theorem \ref{thm:dclass},
the same argument for the map
\[
[\tilde{Q}_{V_A}, \tilde{Q}_I]\hat{\otimes}(I_\TAT\otimes -) 
: \sqK(C_{\mathbb{C}}(\OA), SC_0(\mathbb{T}, 1)\otimes\OA)
\to \sqK(\mathbb{C}, SC_0(\mathbb{T}, 1)\otimes\TAT\otimes\OA)
\]
shows
\begin{align*}
&(\sqK(C_{\mathbb{C}}(\OA), SC_0(\mathbb{T}, 1)\otimes \OA), 
[C((\gamma^A)^{-1}), C(l)])\\
\cong&(\sqK(\mathbb{C}, SC_0(\mathbb{T}, 1)\otimes \TAT\otimes\OA), 
[({\rm id}_\TAT\otimes((\gamma^A)^{-1})_z)(Q_{V_A}(t)), Q_{V_A}(t)]).
\end{align*}
By the definition of $V_A$,
one has 
$((\gamma_\TAT)_z\otimes{\rm id}_\OA)(Q_{V_A}(t))
=({\rm id}_\TAT\otimes((\gamma^A)^{-1})_z)(Q_{V_A}(t))$ 
and this completes the proof of Theorem \ref{thm:gaugeinpi1whatOA}.
\end{proof}

\section{Examples of the reciprocal duals} \label{sec:Examples}
Let $\OA$ be the simple Cuntz--Krieger algebra for an irreducible non-permutation
 matrix $A$ with entries in $\{0,1\}$.
 As in Proposition \ref{prop:characterizationwhatOA},
 the reciprocal dual
 $\widehat{\O}_A$ is not able to be expressed as $\OB$ 
 for any finite square matrix $B$ with entries in $\{0,1\}$.

\medskip

\noindent
{\bf 1.} $\widehat{\O}_2 =\OI,\quad \widehat{\O}_\infty = \O_2.$

Since
$\K_0(\OI) = \Z, \K_1(\OI) = 0$ and hence
\begin{align*}
\Extw^0(\OI) & = \K^0(\OI) = \Free(\K_0(\OI)) = \Z, \\
\Extw^1(\OI) & = \K^1(\OI) = \Tor(\K_0(\OI)) = 0, 
\end{align*}
by using \eqref{eq:sixtermExt} for $\A = \OI$,
we have an exact sequence
\begin{equation*}
0 \longrightarrow
\Exts^0(\OI) \longrightarrow
\Z \longrightarrow
\Z \longrightarrow
\Exts^1(\OI) {\longrightarrow} 0.
\end{equation*}
By \cite{PP}, we know that $\Exts^1(\OI) =0$ so that 
$\Exts^0(\OI) =0$.
Therefore we have
\begin{align*}
\K_0(\OI) = & \Z = \Exts^1(\O_2), \qquad
\K_1(\OI) =  0 = \Exts^0(\O_2), \\
\Exts^1(\OI) & = 0 = \K_0(\O_2), \qquad
\Exts^0(\OI)  = 0 = \K_1(\O_2),
\end{align*}
and hence $\widehat{\O}_2 =\OI,\quad \widehat{\O}_\infty = \O_2.$

\medskip

\noindent
{\bf 2.} $\widehat{\O}_{N+1} =\OI\otimes M_N(\bbC),
\quad (\OI\otimes M_N(\bbC))^{\widehat{}} = \O_{N+1}.$

Similarly to {\bf 1}, we have that 
$\K_0(\OI\otimes M_N(\bbC)) = \Z, \K_1(\OI\otimes M_N(\bbC)) = 0$ and hence
\begin{equation*}
\Extw^0(\OI\otimes M_N(\bbC)) = \Z, \qquad
\Extw^1(\OI\otimes M_N(\bbC))  = 0, 
\end{equation*}
so that we know  an exact sequence
\begin{equation*}
0 \longrightarrow
\Exts^0(\OI\otimes M_N(\bbC)) {\longrightarrow}
\Z {\longrightarrow}
\Z {\longrightarrow}
\Exts^1(\OI\otimes M_N(\bbC)) {\longrightarrow} 0.
\end{equation*}
By \cite{PP}, we know that $\Exts^1(\OI\otimes M_N(\bbC)) =\Z/N\Z$ 
so that 
$\Exts^0(\OI\otimes M_N(\bbC)) =0$.
Therefore we have
\begin{align*}
\K_0(\OI\otimes M_N(\bbC)) = & \Z = \Exts^1(\O_{N+1}), \qquad
\K_1(\OI\otimes M_N(\bbC)) =  0 = \Exts^0(\O_{N+1}), \\
\Exts^1(\OI\otimes M_N(\bbC)) & = \Z/N\Z = \K_0(\O_{N+1}), \qquad
\Exts^0(\OI\otimes M_N(\bbC))  = 0 = \K_1(\O_{N+1}),
\end{align*}
and hence $\widehat{\O}_{N+1} =\OI\otimes M_N(\bbC),
(\OI\otimes M_N(\bbC))^{\widehat{}} = \O_{N+1}.$

\medskip

\noindent
{\bf 3.} $((\OI)_{1-S_1 S_1^*})^{\widehat{}} =\OA,
\quad \widehat{\O}_{A} =(\OI)_{1-S_1 S_1^*},$
where
$
A =
\begin{bmatrix}
1 & 1 & 1 & 1 & 1 \\
0 & 1 & 1 & 1 & 0 \\
1 & 1 & 1 & 1 & 1 \\
0 & 1 & 1 & 1 & 0 \\
1 & 1 & 1 & 1 & 1 
\end{bmatrix}.
$

Put $\B =(\OI)_{1-S_1 S_1^*}$.
We then know that 
$\K_0(\B)\cong \K_0(\OI) =\Z$
and
$\K_1(\B)\cong \K_1(\OI) =0.$
It is easy to see that $[1_\B]_0$ in $\K_0(\B)$ is $0$ in $\Z.$
Hence by Proposition \ref{prop:characterizationwhatOA},
one may find a simple Cuntz--Krieger algebra $\A$
such that $\widehat{\A} \cong \B$.
Now $[1_\B]_0 =0$ in $\K_0(\B)= Z,$
so that $[1_\B]_0$ is torsion in $\K_0(\B)$
and $n_\B=0$ in the proof of Proposition \ref{prop:equivalentwhatOA}.
This shows that the map
$\partial_\B: \Z \rightarrow \K_0(\B)$ sending $1$ to  $[1_\B]_0$
is the zero-map. 
Hence we have short exact sequences:
\begin{gather*}
0 \longrightarrow \K_1(\B)(=0) 
\longrightarrow \K_1(\A)
\longrightarrow \Z
\longrightarrow 0, \\
0 \longrightarrow \partial_\B(\Z)(=0) 
\longrightarrow \K_0(\B)
\longrightarrow \K_0(\A)
\longrightarrow 0
\end{gather*}
so that 
\begin{equation*}
\K_1(\A) = \Z, \qquad \K_0(\A) = \K_0(\B) =\Z.
\end{equation*}
As
$\widehat{\A} \cong \B$, we have
\begin{equation*}
\Exts^1(\A) = \K_0(\B) =Z, \qquad
\Exts^0(\A) = \K_1(\B) =0.
\end{equation*}
On the other hand, for the  matrix
$
A =
\begin{bmatrix}
1 & 1 & 1 & 1 & 1 \\
0 & 1 & 1 & 1 & 0 \\
1 & 1 & 1 & 1 & 1 \\
0 & 1 & 1 & 1 & 0 \\
1 & 1 & 1 & 1 & 1 
\end{bmatrix},
$
we have
$
\widehat{A} =
\begin{bmatrix}
1 & 1 & 1 & 1 & 1 \\
0 & 1 & 1 & 1 & 0 \\
0 & 0 & 0 & 0 & 0 \\
0 & 1 & 1 & 1 & 0 \\
0 & 0 & 0 & 0 & 0 
\end{bmatrix}.
$
It is straightforward to see that  
\begin{equation*}
\Z^5/(I -A)\Z^5 \cong \Z, \qquad 
\Ker(I -A) \text{ in } \Z^5 \cong \Z, \qquad
\Z^5/(I -\widehat{A})\Z^5 \cong \Z, 
\end{equation*}
which show that
\begin{equation*}
\K_0(\OA) \cong \Z, \qquad
\K_1(\OA) \cong \Z, \qquad
\Exts^1(\OA) \cong \Z. 
\end{equation*}
By Proposition \ref{prop:basic2},
we have that $\OA$ is isomorphic to $\A$,
showing that $\widehat{\O}_A \cong \B$.

\medskip

{\it Acknowledgments:}
K. Matsumoto is supported by JSPS KAKENHI Grant Number 24K06775.
T. Sogabe is supported by JSPS KAKENHI Grant Number 24K16934.

\end{document}